\documentclass[12pt,reqno]{amsart}
\usepackage{amsmath,amsfonts,amsthm,graphicx}
\usepackage{a4}

\usepackage{graphicx}

\usepackage{tikz}
\usetikzlibrary{arrows}
\usetikzlibrary{decorations.markings}

\newtheorem{theorem}{Theorem}[section]
\newtheorem{lemma}[theorem]{Lemma}
\newtheorem{corollary}[theorem]{Corollary}

\theoremstyle{definition}

\theoremstyle{remark}
\newtheorem{remark}[theorem]{Remark}

\numberwithin{equation}{section}

\usepackage{color}
\newcommand{\cred}[1]{#1}   %barva
\renewcommand{\b}[1]{\boldsymbol{#1}}
\newcommand{\ba}{\b{a}}
\newcommand{\bfx}{\b{x}}
\newcommand{\bn}{\b{n}}
\newcommand{\bq}{\b{q}}
\newcommand{\cA}{\mathcal{A}}
\newcommand{\cEN}{\mathcal{E}^\mathrm{N}}
\newcommand{\cENO}{\mathcal{E}^{\mathrm{N}0}}
\newcommand{\cENP}{\mathcal{E}^{\mathrm{N}+}}
\newcommand{\cK}{\mathcal{K}}
\newcommand{\cL}{\mathcal{L}}
\newcommand{\cM}{\mathcal{M}}
\newcommand{\cT}{\mathcal{T}}
\newcommand{\ddiv}{\operatorname{div}}
\newcommand{\dx}[1][x]{\,\mathrm{d}#1}
\newcommand{\Erelestn}{E_{\mathrm{rel},n}^\mathrm{est}} %{\mathrm{RelErr}}
\newcommand{\GammaD}{{\Gamma_{\mathrm{D}}}}
\newcommand{\GammaN}{{\Gamma_{\mathrm{N}}}}
\newcommand{\Hdiv}[1][\Omega]{\boldsymbol{H}(\ddiv,#1)}
\newcommand{\ihom}{m}
\newcommand{\interior}{\operatorname{int}}
\newcommand{\MM}{M}

\newcommand{\Ninf}{{N_\infty}}
\newcommand{\norm}[1]{\|#1\|}
\newcommand{\oGammaD}{\overline\Gamma_{\mathrm{D}}}

\newcommand{\R}{\mathbb{R}}
\newcommand{\RT}{\mathbf{RT}_p}
\newcommand{\linspan}{\operatorname{span}}

\newcommand{\VT}{V_h}
\newcommand{\vT}{v_h}

\newcommand{\tqa}{{\b{q}}_h^{\ba}}
\newcommand{\tqT}{\b{q}_h}
\newcommand{\bWa}{\b{W}_{\ba}}
\newcommand{\bWT}{\b{W}_h}
\newcommand{\bwT}{\b{w}_h}

\newcommand{\cETK}{\mathcal{E}_h^K}
\newcommand{\cEaI}{\mathcal{E}^{\mathrm{I}}_{\ba}}
\newcommand{\cEaBE}{\mathcal{E}^{\mathrm{B,E}}_{\ba}}
\newcommand{\cEaBD}{\mathcal{E}^{\mathrm{B,D}}_{\ba}}
\newcommand{\cEaBN}{\mathcal{E}^{\mathrm{B,N}}_{\ba}}

\newcommand{\da}{d_h^{\ba}}
\newcommand{\NT}{\mathcal{N}_h}
\newcommand{\NTD}{\mathcal{N}_h^{\mathrm{D}}}
\newcommand{\NTI}{\mathcal{N}_h^{\mathrm{I}}}
\newcommand{\NTK}{\mathcal{N}_h^K}
\newcommand{\NTN}{\mathcal{N}_h^{\mathrm{N}}}
\newcommand{\NTGa}{\mathcal{N}^{\cT}_{\gamma}}
\newcommand{\HF}{\psi_{\ba}}
\newcommand{\Ta}{\cT_{\ba}}
\newcommand{\Oma}{\omega_{\ba}}
\newcommand{\poma}{\partial\Oma}
\newcommand{\rTa}{{\tilde r_h^{\ba}}}
\newcommand{\gTa}{\tilde g_h^{\ba}}
\newcommand{\PpTast}{P_{p}^{*}(\Ta)}
\newcommand{\PpT}{P_p(\Ta)}
\newcommand{\tsa}{{\b{s}}_h^{\ba}}

\begin{document}

\title[New guaranteed lower bounds on eigenvalues]{New guaranteed lower bounds on eigenvalues by conforming finite elements}


%    author one information
\author{Tom\'a\v{s} Vejchodsk\'y}
\address{Tom\'a\v{s} Vejchodsk\'y, Institute of Mathematics, Czech Academy of Sciences, \v{Z}itn\'a 25, CZ-115 67 Praha 1, Czech
Republic}
%\curraddr{}
\email{vejchod@math.cas.cz}
\thanks{T.\ Vejchodsk\'y gratefully acknowledges the support of Neuron Fund for Support of Science, project no.~24/2016, and the institutional support RVO~67985840.}

%    author two information
\author{Ivana \v{S}ebestov\'a}
\address{Ivana \v{S}ebestov\'a, Departamento de Ingenier{\'\i}a Matem\'atica, Facultad de Ciencias F{\'\i}sicas y Matem\'aticas, Universidad de Concepci\'on, Casilla 160-C, Concepci\'on, Chile}
%\curraddr{}
\email{isebestova@udec.cl}
\thanks{I.\ \v{S}ebestov\'a acknowledges the support of a Fondecyt Postdoctoral Grant no.~3150047 and of Anillo ANANUM, ACT1118, Comisi\'on Nacional de Investigaci\'on Cient\'ifica y
Tecnol\'ogica (Chile).}

%    \subjclass is required.

\date{}

\dedicatory{}

%    Abstract is required.
\begin{abstract}
We provide two new methods for computing lower bounds of eigenvalues of symmetric elliptic second-order differential operators with mixed boundary conditions of Dirichlet, Neumann, and Robin type. The methods generalize ideas of Weinstein's and Kato's bounds and they are designed for a simple and straightforward implementation in the context of the standard finite element method. These lower bounds are obtained by a posteriori error estimators based on local flux reconstructions, which can be naturally utilized for adaptive mesh refinement.
We derive these bounds, prove that they estimate the exact eigenvalues from below, and
illustrate their practical performance by a numerical example.
\end{abstract}

\maketitle


\section{Introduction}
\label{se:intro}

We consider the problem of finding
eigenvalues $\lambda_i\in\R$ and nonzero eigenfunctions $u_i \in V$, $i=1,2,\dots$, such that
\begin{equation}
  \label{eq:EP1}
  a(u_i,v) = \lambda_i b(u_i, v) \quad \forall v \in V,
\end{equation}
where $a$ is the inner product in the real Hilbert space $V$ and $b$ is a symmetric positive semidefinite bilinear form on $V$, see Section~\ref{se:lbounds} for details.
Considering arbitrary numbers $\lambda_{*,i} \in \R$ and functions $u_{*,i} \in V$, $b(u_{*,i}, u_{*,i}) = 1$, which are supposed to approximate eigenpairs $\lambda_i$, $u_i$, we introduce representatives $w_i \in V$ of residuals uniquely determined by the identity
\begin{equation}
\label{eq:defw}
a(w_i,v) = a(u_{*,i},v) - \lambda_{*,i} b(u_{*,i},v) \quad\forall v \in V.
\end{equation}
Representatives $w_i$ cannot be determined exactly in general, but we can compute estimates of their energy norm.
In the context of the finite element method for symmetric elliptic second-order differential operators, we can efficiently reconstruct the flux $\nabla u_{*,i}$ and compute an accurate guaranteed upper bound $\eta_i$ on the energy norm of $w_i$:
\begin{equation}
\label{eq:abscompl}
  \norm{w_i}_a \leq \eta_i.  %\cblue{\quad \norm{{\cdot}}_V \text{ not defined.}}
\end{equation}
This enables to compute the following explicit lower bounds on the eigenvalue $\lambda_n$:
\begin{align}
\label{eq:Wein}
  %\ell_i \leq \lambda_i, \quad\text{where }
  %\ell_i &= \frac14 \left( -\frac{\eta}{|u_*|_b} + \sqrt{\frac{\eta^2}{{|u_*|_b^2}} + 4\lambda_*} \right)^2,
  %\ell_n &= \frac{1}{4|u_{*,n}|_b^2} \left( -\eta_n + \sqrt{\eta_n^2 + 4\lambda_{*,n}|u_{*,n}|_b^2} \right)^2,
  \ell_n &= \frac{1}{4} \left( -\eta_n + \sqrt{\eta_n^2 + 4\lambda_{*,n}} \right)^2,
\\
  \label{eq:Kato}
  %L_i \leq \lambda_i, \quad\text{where }
  %L_n &= \lambda_* \left( 1 + \frac{\nu}{\lambda_*(\nu - \lambda_*)} \frac{\eta^2}{|u_*|_b^2} \right)^{-1},
  L_n &= \lambda_{*,n} \left( 1 + \nu \lambda_{*,n} \sum_{i=n}^s \frac{\eta_i^2}{\lambda_{*,i}^2 (\nu - \lambda_{*,i} ) } \right)^{-1},
\end{align}
where $\nu$ is assumed to satisfy $\lambda_{*,s} < \nu$ for some index $s \geq n$. Quantity $\ell_n$ is proved to be below $\lambda_n$ if $\lambda_{n-1}\lambda_n \leq \lambda_{*,n}^2 \leq \lambda_n \lambda_{n+1}$, see Theorem~\ref{th:Weinlowerbound} below. Similarly, $L_n$ is a lower bound on $\lambda_n$ if $\lambda_{*,s} < \nu \leq \lambda_{s+1}$ and if $\lambda_{*,i}$, $u_{*,i}$, $i=n, n+1, \dots, s$, solve a discrete eigenvalue problem, see Theorem~\ref{th:Katobound} below.

Lower bound $\ell_n$ is quite general and robust, but its convergence in terms of $\eta_n$ is linear, which is suboptimal. Lower bound $L_n$ converges with the optimal quadratic speed, but it requires an additional parameter $\nu$ to be chosen, which may be complicated in some cases. % Therefore, we propose to compute both these bounds and use the more accurate one.
%The bound \eqref{eq:Kato} is a simplified version of the bound we prove below, see \eqref{eq:Katobound}.
%The simplified version \eqref{eq:Kato} is efficient if the exact eigenvalues $\lambda_n$ and $\lambda_{n+1}$ are sufficiently separated. The version provided in \eqref{eq:Katobound} is efficient even in the case of multiple and clustered eigenvalues.
The accuracy of $L_n$ depends on the size of $\nu - \lambda_{*,s}$ and this size is limited by the spectral gap $\lambda_{s+1}-\lambda_s$. Therefore we try to choose $s \geq n$ such that the spectral gap $\lambda_{s+1}-\lambda_s$ is relatively large.

Lower bounds $\ell_n$ and $L_n$ are inspired by Weinstein's \cite[Corollary 6.20]{Chatelin1983} and Kato's \cite{Kato1949} bounds, respectively. However, Weinstein's and Kato's bounds are standardly formulated for an abstract operator on a Hilbert space and it is not immediately clear, how to use them in practical computations, especially if the trial functions do not posses extra regularity such that the elliptic operator can be applied pointwise.
Therefore, we modify these classical results in a nontrivial way such that they can be straightforwardly used for the eigenvalue problem \eqref{eq:EP1} and consequently in the context of the finite element method with the standard regularity of trial functions.
The finite element method as a special case of the Galerkin method is well known to provides guaranteed upper bounds on eigenvalues and, thus, in combination with bounds $\ell_n$ and $L_n$ we obtain two-sided bounds and full control of accuracy. For a nice survey of existing eigenvalue bounds including Weinstein's and Kato's bounds and for their generalizations we refer to \cite{Plum1997}.

The difficult problem of lower bounds of eigenvalues have already been studied for decades and several different approaches were developed. Recently, the nonconforming finite element methods were proposed \cite{AndRac:2012,HuHuaLin2014,HuHuaShe2015,LinLuoXie2014,LinXieLuoLiYan:2010,LiuOis2013,LuoLinXie:2012,Rannacher:1979,YanZhaLin:2010}.
These approaches provide typically an asymptotic lower bound in the sense that the lower bound is guaranteed only if the corresponding discretization mesh is sufficiently fine. Guaranteed lower bound for Laplace eigenvalues with homogeneous Dirichlet boundary conditions are obtained even on coarse meshes in \cite{CarGed2014} by using the Crouzeix--Raviart nonconforming finite elements.
Paper \cite{Liu2015} improves this result by removing the separation condition for higher eigenvalues.
A generalization of these ideas to a biharmonic operator is provided in \cite{CarGal2014}.

A lower bound on the smallest eigenvalue is obtained in \cite{Repin2012} by a nonoverlapping decomposition of the domain into subdomains, where the exact eigenvalues are known.
In a sense similar method is proposed in \cite{Kuz_Rep_Guar_low_bound_smal_eig_elip_13}. It is based on an overlapping decomposition of the domain into geometrically simple subdomains and it yields a lower bound on the smallest eigenvalue for homogeneous Neumann or mixed Neumann-Dirichlet boundary conditions. A lower bound on the smallest eigenvalue for a triangle is obtained in \cite{Kobayashi2015} using a scaling. An interesting generalization of the method of eigenvalue inclusions \cite{BehnkeGoerish1994,Plum1991} for the Maxwell operator is provided in \cite{Barrenechea2014}. Reference \cite{GruOva2009} includes both estimators of eigenvalues and corresponding eigenfunctions.

In \cite{Seb_Vejch_2sidedb_eigen_Fr_Poin_trace_14} we propose another approach based on a combination of the method of a priori-a posteriori inequalities \cite{Sigillito:1977,KutSig:1978} and a complementarity technique \cite{Complement:2010,systemaee:2010}. The method in \cite{Seb_Vejch_2sidedb_eigen_Fr_Poin_trace_14} yields a linearly convergent lower bound on the principal eigenvalue and it is based on a global flux reconstructions. Bounds $\ell_n$ and $L_n$ given by \eqref{eq:Wein} and \eqref{eq:Kato}, respectively, improve this result in several aspects. They are applicable for arbitrary eigenvalues, the bound $L_n$ is quadratically convergent, and the flux is reconstructed locally, which enables an efficient and naturally parallel implementation. Note that the used local flux reconstruction was originally proposed in \cite{BraSch:2008} for source problems and we modify it for eigenvalue problems.

The rest of the paper is organized as follows.
Section~\ref{se:lbounds} presents Weinstein's and Kato's method in a general weak setting and proves the lower bounds.
Section~\ref{se:ellop} introduces a particular eigenvalue problem for a symmetric second-order elliptic differential operator, proves its well posedness and briefly describes its finite element discretization.
Section~\ref{se:compl} derives a general guaranteed upper bound \eqref{eq:abscompl} on the representative of the residual based on the complementarity approach.
Section~\ref{se:flux} defines the local flux reconstruction and proves its properties.
Section~\ref{se:numex} illustrates the practical performance of the new lower bounds using the dumbbell shaped domain example.
Finally, Section~\ref{se:concl} draws the conclusions.

\section{Lower bounds in the abstract setting}
\label{se:lbounds}

This section briefly describes the rigorous mathematical setting of the eigenvalue problem~\eqref{eq:EP1}. The generality of this setting enables to treat the standard types of eigenvalue problems such as the Dirichlet, Neumann, Steklov, etc. in a unified manner.

Let $V$ be a real Hilbert space with a scalar product $a(u,v)$ for $u,v\in V$. In particular the form $a(u,v)$ is continuous, bilinear, symmetric and positive definite. Further, let a form $b(u,v)$ for $u,v\in V$ be continuous, bilinear, symmetric, and positive semidefinite, i.e. $b(v,v) \geq 0$ for all $v \in V$. We use notation $\norm{v}_a^2 = a(v,v)$ and $|v|_b^2 = b(v,v)$ for the norm induced by the scalar product $a$ and the seminorm induced by the bilinear form $b$, respectively.
We assume that the seminorm $|\cdot|_b$ is compact with respect to the norm $\norm{\cdot}_a$, i.e. from any sequence bounded in $\norm{\cdot}_a$, we can extract a subsequence which is Cauchy in $|\cdot|_b$.
Under these assumptions we consider the eigenvalue problem \eqref{eq:EP1} and
% find eigenvalues $\lambda_i\in\R$ and nonzero eigenfunctions $u_i \in V$ such that
% \begin{equation}
%   \label{eq:EP1}
%   a(u_i,v) = \lambda_i b(u_i, v) \quad \forall v \in V.
% \end{equation}
summarize its properties in the following theorem.

\begin{theorem}
\label{th:properties}
% \cred{
% A PROBLEM IS THAT THE SEQUENCE OF EIGENVALUES CAN BE FINITE OR EVEN EMPTY (IF $S=0$).
% SOLUTION: BASED ON \eqref{eq:MKa-orth}, WE CAN EASILY SHOW THAT $\cM = \cK^\perp = \{ u \in V : a(u,v) = 0 \forall v \in \cK \}$. THUS, WE NEED TO ASSUME THAT $\dim \cK^\perp = \infty$ AND USE THE STATEMENT OF THE HILBERT-SCHMIDT THEOREM SAYING THAT THE SET OF NONZERO EIGENVALUES IS AT MOST COUNTABLE.
% }
Under the above assumptions, problem \eqref{eq:EP1} has the following properties.
\begin{itemize}
\item[(a)]
There exists \cred{(at most)} countable \cred{(and possibly empty)} sequence of eigenvalues
$$
  0 < \lambda_1 \leq \lambda_2 \leq \lambda_3 \leq \cdots
  %\quad \lim_{i\rightarrow\infty} \lambda_i = \infty
$$
and the corresponding eigenfunctions can be normalized as
\begin{equation}
\label{eq:normalization}
  b(u_i,u_j) = \delta_{ij}, \quad \forall i,j=1,2,\dots,\cred{\Ninf,}
\end{equation}
\cred{where $\Ninf$ is the number of eigenvalues. If it is infinite, we set $\Ninf=\infty$.}
\item[(b)]
The space $V$ can be decomposed as
\begin{equation}
\label{eq:decomposition}
  V = \cM \oplus \cK,
\end{equation}
where $\cM = \operatorname{span} \{ u_1, u_2, \dots \}$ \cred{is the linear span of all eigenfunctions, $\operatorname{dim} \cM = \Ninf$}, and $\cK = \{v \in V : |v|_b = 0 \}$.
Consequently, any function $v\in V$ can be uniquely decomposed as
\begin{equation}
  \label{eq:vMvK}
  v = v^\cM + v^\cK, \quad\text{where } v^\cM \in \cM \text{ and } v^\cK \in \cK.
\end{equation}
\item[(c)]
The decomposition \eqref{eq:decomposition} satisfies
\begin{align}
  \label{eq:MKa-orth}
  a(v,u) &= 0 \quad\forall v \in \cM,\ \forall u\in \cK, \\
  \label{eq:VKb-orth}
  b(v,u) &= 0 \quad\forall v \in V,\ \forall u \in \cK.
\end{align}
\item[(d)]
Any function $v\in V$ satisfies
\begin{align}
  \label{eq:Parseval}
  |v|_b^2 &= \sum\limits_{i=1}^\Ninf |b(v,u_i)|^2,
\\
  \label{eq:a-norm}
  \norm{v}_a^2 &= \norm{v^\cM}_a^2 + \norm{v^\cK}_a^2 \quad\text{ with }
    \norm{v^\cM}_a^2 = \sum\limits_{i=1}^\Ninf \lambda_i |b(v,u_i)|^2,
\end{align}
where $v^\cM \in \cM$ and $v^\cK \in \cK$ are given by \eqref{eq:vMvK}.
\end{itemize}
\end{theorem}
\begin{proof}
To prove (a), we first mention that whenever $\lambda_i$, $u_i$ is an eigenpair of \eqref{eq:EP1} then $\lambda_i >0$ and $|u_i|_b > 0$, because $0 < \norm{u_i}_a^2 = \lambda_i |u_i|_b^2$.
In order to use the spectral theory of compact operators, we define the solution operator $S: V \rightarrow V$ by the identity
\begin{equation}
\label{eq:solop}
  a(Su,v) = b(u,v) \quad \forall v \in V.
\end{equation}
It is an elementary exercise to prove that the compactness of $|\cdot|_b$ with respect to $\norm{\cdot}_a$ is equivalent to the compactness of the solution operator $S$. Thus, the Hilbert--Schmidt spectral theorem \cite[Theorem~4, Chapter~II, section~3]{Gaal_Lin_anal_repres_theo_73} applied to $S$ provides the existence of the countable sequence of eigenvalues and the corresponding orthogonal system of eigenfunctions.

Statement (b) is another consequence of the Hilbert--Schmidt spectral theorem, because it claims that $V = \cM \oplus \ker S$, where $\ker S = \{ v \in V : Sv=0 \}$ is the kernel of $S$. Thus, it remains to show the equality $\cK = \ker S$. If $u \in \cK$ then
\begin{equation}
  \label{eq:VKb-orth-proof}
  |b(u,v)| \leq |u|_b |v|_b = 0 \quad \forall v \in V.
\end{equation}
Thus, $a(Su,v) = b(u,v) = 0$ for all $v\in V$ and $Su = 0$. On the other hand, if $u\in \ker S$ then $0 = a(Su,u) = b(u,u) = |u|_b^2$ and $u \in \cK$.

To prove identity \eqref{eq:MKa-orth} in (c), we express $v \in\cM$ as $v = \sum_{i=1}^\Ninf c_i u_i$ and proceed as follows:
$$
  a(v,u) = \sum_{i=1}^\Ninf c_i a(u_i, u) = \sum_{i=1}^\Ninf c_i \lambda_i b(u_i, u)
= \sum_{i=1}^\Ninf c_i \lambda_i a(u_i, Su) = 0,
$$
where the last equality holds, because $u \in \cK = \ker S$.
Identity \eqref{eq:VKb-orth} has already been proved in \eqref{eq:VKb-orth-proof}.

Finally, the equalities in (d) follow from the splitting \eqref{eq:vMvK}.
Since $b(v,u^\cK) = a(v, S u^\cK) = 0$ for all $u^\cK \in \cK = \ker S$ and all $v \in V$,
we obtain $|v|_b = |v^\cM|_b$. The expansion $v^\cM = \sum_{i=1}^\Ninf c_i u_i$ with $c_i = b(v,u_i) = a(v,u_i)/\lambda_i$ and the Parseval's identity $|v^\cM|_b^2 = \sum_{i=1}^\Ninf c_i^2$ yields \eqref{eq:Parseval}.
Similarly, by \eqref{eq:MKa-orth} we obtain $\norm{v}_a^2 = \norm{v^\cM}_a^2 + \norm{v^\cK}_a^2$
and easy computation gives $\norm{v^\cM}_a^2 = \sum_{i=1}^\Ninf c_i^2 \lambda_i$, which shows \eqref{eq:a-norm}.
\end{proof}

Properties listed in Theorem~\ref{th:properties} enable to prove an estimate of a certain distance between an arbitrary number $\lambda_{*,n} \in \R$ and the exact spectrum.
\begin{theorem}
\label{th:estwa}
Let $\lambda_i$, $i=1,2,\dots$, be eigenvalues of \eqref{eq:EP1}.
Let $u_{*,n} \in V\setminus\{0\}$ and $\lambda_{*,n}\in \R$ be arbitrary.
Consider $w_n\in V$ given by \eqref{eq:defw}.
% \begin{equation}
% \label{eq:defw}
% a(w,v) = a(u_*,v) - \lambda_* b(u_*,v) \quad\forall v \in V.
% \end{equation}
Then
\begin{equation}
  \label{eq:disttospectrum}
  \min\limits_i \frac{|\lambda_i-\lambda_{*,n}|^2}{\lambda_i}
   \leq \frac{\norm{w_n}_a^2}{|u_{*,n}|_b^2}.
\end{equation}
\end{theorem}
\begin{proof}
By \eqref{eq:a-norm} and \eqref{eq:defw} we express
\begin{multline}
\label{eq:anormwident}
  \norm{w_n}_a^2 - \norm{w_n^\cK}_a^2
    = \sum\limits_{i=1}^\Ninf \lambda_i |b(w_n,u_i)|^2
    = \sum\limits_{i=1}^\Ninf \frac{|a(w_n,u_i)|^2}{\lambda_i}
  \\
    = \sum\limits_{i=1}^\Ninf \frac{|a(u_{*,n},u_i) - \lambda_{*,n} b(u_{*,n},u_i)|^2}{\lambda_i}
    = \sum\limits_{i=1}^\Ninf \frac{|\lambda_i - \lambda_{*,n}|^2}{\lambda_i} |b(u_{*,n},u_i)|^2.
\end{multline}
Thus, %since $\norm{w}_a^2 \geq \norm{w}_a^2 - \norm{w^\cK}_a^2$, we obtain
$$
  \norm{w_n}_a^2 \geq
    \min_i \frac{|\lambda_i - \lambda_{*,n}|^2}{\lambda_i} \sum\limits_{i=1}^\Ninf |b(u_{*,n},u_i)|^2
$$
and Parseval's identity \eqref{eq:Parseval} immediately yields the statement \eqref{eq:disttospectrum}.
\end{proof}

Next we prove that Theorem~\ref{th:estwa} yields a lower bound for the eigenvalue $\lambda_n$, provided
an approximation $\lambda_{*,n}$ is not far away from $\lambda_n$ in the sense that
\begin{equation}
\label{eq:closest}
  \sqrt{\lambda_{n-1} \lambda_n} \leq \lambda_{*,n}
  %\quad\text{and}\quad \lambda_*
  \leq \sqrt{\lambda_n \lambda_{n+1}}.
\end{equation}

% The following theorem is a consequence of Theorem~\ref{th:estwa} and it has been proved in~\cite[Theorem~3.4]{Seb_Vejch_2sidedb_eigen_Fr_Poin_trace_14} in a slightly different context. We repeat it here, because it provides an abstract enclosure on the principal eigenvalue of~\eqref{eq:EP1} and we use it below in Theorem~\ref{th:lowerbound} to obtain the lower bound for the general symmetric elliptic eigenvalue problem.
\begin{theorem}
\label{th:Weinlowerbound}
Let $u_{*,n}\in V$, $|u_{*,n}|_b = 1$, be arbitrary and let $\lambda_{*,n} \in \R$ satisfies
the closeness condition \eqref{eq:closest}.
%for a fixed index $n$.
Further, let there be $\eta_n \in \R$ such that $w_n\in V$ given by \eqref{eq:defw} satisfies \eqref{eq:abscompl}.
% \begin{equation}
% \label{eq:abscompl}
%   \norm{w}_a \leq \eta  %\cblue{\quad \norm{{\cdot}}_V \text{ not defined.}}
% \end{equation}
Then $\ell_n$ defined in \eqref{eq:Wein} satisfies
\begin{equation}
\label{eq:Weinbound}
  \ell_n \leq \lambda_n.
  %\quad\text{where } \ell_i = \frac14 \left( -\frac{\eta}{|u_*|_b} + \sqrt{\frac{\eta^2}{{|u_*|_b^2}} + 4\lambda_*} \right)^2.
  %C_{ab} &\leq& 1/X_2 \equiv \Cgup \label{eq:enclo_Fried_Poinc}
\end{equation}
\end{theorem}
\begin{proof}
For brevity, let us put $\lambda_* = \lambda_{*,n}$ and $u_*=u_{*,n}$.
Let us notice that condition \eqref{eq:closest} implies that $\lambda_i \lambda_n \leq \lambda_*^2$ for all $i=1,2,\dots,n-1$ and that $\lambda_*^2 \leq \lambda_n \lambda_i$ for all $i=n+1, n+2, \dots, \cred{\Ninf}$.
Consequently, inequality $\lambda_*^2(\lambda_i - \lambda_n) \leq \lambda_n \lambda_i (\lambda_i - \lambda_n)$
holds true for all $i=1,2,\dots, \cred{\Ninf}$. This inequality is however equivalent to
$$
  \frac{(\lambda_n - \lambda_*)^2}{\lambda_n} \leq \frac{(\lambda_i - \lambda_*)^2}{\lambda_i}
  \quad \forall i=1,2,\dots, \cred{\Ninf}.
$$
Hence, using this estimate in \eqref{eq:disttospectrum}, bound \eqref{eq:abscompl}, and assumption $|u_{*,n}|_b = 1$,
we immediately obtain
$$
 \frac{(\lambda_n - \lambda_*)^2}{\lambda_n} \leq \eta_n^2. % \frac{\eta_n^2}{|u_*|_b^2}. %\eta^2.
$$
This can be rewritten as the quadratic inequality $\lambda_n^2 - \bigl(2\lambda_* + \eta_n^2 \bigr)\lambda_n + \lambda_*^2 \leq 0$ and this inequality can only be satisfied if \eqref{eq:Weinbound} holds true.
\end{proof}

Lower bound $\ell_n$ is quite universal and robust, but its convergence with respect to $\eta_n$ is linear, which is suboptimal.
% Indeed,
% %setting $E = \eta / |u_*|_b$ for brevity,
% we have $\lambda_* - \ell_n = \eta\bigl(-\eta + \sqrt{\eta^2 + 4\lambda_* |u_*|_b^2}\bigr)/(2|u_*|_b^2) \geq 0$ and since
% $\eta^2 + 4\lambda_* |u_*|_b^2 \leq \bigl(\eta + 2\sqrt{\lambda_*} |u_*|_b \bigr)^2$, we obtain
% $$
%   0 \leq \lambda_* - \ell_n \leq \frac{\eta}{|u_*|_b} \sqrt{\lambda_*}.
% $$
% !!! NEED ALSO: $ C \eta \leq \lambda_* - \ell_n$.
% EVEN BETTER: $ C_1 \eta \leq |\lambda_n - \ell_n| \leq C_2 \eta$. !!!
However, inspired by Kato's bound \cite{Kato1949} we are able to derive a quadratically convergent lower bound. Before we do it, we introduce an auxiliary lemma.
\begin{lemma}
Let $u_* \in V$ and $\lambda_* \in \R$ be arbitrary and let $w\in V$ be given by
\begin{equation}
\label{eq:defw*}
  a(w,v) = a(u_*,v) - \lambda_* b(u_*,v) \quad\forall v \in V,
\end{equation}
see \eqref{eq:defw}.
Then components $u_*^\cK$ and $w^\cK$ defined in \eqref{eq:vMvK} satisfy
\begin{equation}
\label{eq:anormu0w0}
  \norm{u_*^\cK}_a = \norm{w^\cK}_a.
\end{equation}
Moreover,
\begin{equation}
\label{eq:bsqoverlambda}
\sum_{i=1}^\Ninf \frac{|b(u_*, u_i)|^2}{\lambda_i}
  = \frac{1}{\lambda_*^2} \left( \norm{w}_a^2 + 2 \lambda_* |u_*|_b^2 - \norm{u_*}_a^2 \right).
\end{equation}
\end{lemma}
\begin{proof}
Using $v = w^\cK$ in \eqref{eq:defw*} and orthogonalities \eqref{eq:MKa-orth} and \eqref{eq:VKb-orth}, we obtain
$\norm{w^\cK}_a^2 = a(u_*^\cK,w^\cK)$. Similarly, by using $v=u_*^\cK$ in \eqref{eq:defw*}, we have
$a(w^\cK,u_*^\cK) = \norm{u_*^\cK}_a^2$ and equality \eqref{eq:anormu0w0} is proved.

Using \eqref{eq:anormwident}, \eqref{eq:a-norm}, and \eqref{eq:Parseval}, we obtain equality
\begin{multline*}
  \norm{w}_a^2 - \norm{w^\cK}_a^2
  = \sum\limits_{i=1}^\Ninf \left( \lambda_i - 2 \lambda_* + \frac{\lambda_*^2}{\lambda_i} \right)  |b(u_*,u_i)|^2
\\
  = \norm{u_*}_a^2 - \norm{u_*^\cK}_a^2 - 2 \lambda_* |u_*|_b^2
    + \lambda_*^2 \sum\limits_{i=1}^\Ninf \frac{|b(u_*,u_i)|^2}{\lambda_i}.
\end{multline*}
Identity \eqref{eq:bsqoverlambda} now follows by employing \eqref{eq:anormu0w0} and a simple rearrangement.
\end{proof}

\begin{theorem}
\label{th:Katobound}
%Let $ 0 < r \leq s$ be fixed indices.
%\cred{ALTERNATIVA 1: CERVENE: NUTNO DODELAT, ASI NAME PRAKTICKY SMYSL
%Let $r$ and $s$ be fixed indices satisfying $0 < r \leq 1+s/2$.}
Let $\tilde u_{*,i} \in V$ for $i = r, r+1, \dots, s$ be arbitrary.
Let $\widetilde V_* = \linspan \{ \tilde u_{*,r}, \tilde u_{*,r+1}, \dots, \tilde u_{*,s} \}$.
Let $\lambda_{*,i} > 0$ and $u_{*,i} \in \widetilde V_*$, $|u_{*,i}|_b = 1$, $i = r, r+1, \dots, s$,
be the eigenvalues sorted in ascending order and the corresponding eigenfunctions of problem
\begin{equation}
\label{eq:smallEP}
  a( u_{*,i}, v_* ) = \lambda_{*,i} b( u_{*,i}, v_* )
  \quad \forall v_* \in \widetilde V_*.
\end{equation}
Let there exist $\nu > 0$ satisfying
%\cblue{ALTERNATIVA 2 - asi nejlepsi - automaticky splneno pro $r=1$ a $r=2$ - Na co tam to $r$ vlastne je???:}
\begin{equation}
  \label{eq:nucond}
  \cred{\lambda_{s-1} \leq}
  \lambda_{*,s} < \nu \leq \lambda_{s+1},
\end{equation}
where $\lambda_{s+1}$ is the eigenvalue of \eqref{eq:EP1}.
% Further, let $w_j \in V$ be given by \eqref{eq:defw} with $\lambda_* = \lambda_{*,j}$ and $u_* =  u_{*,j}$ and let $\eta_j > 0$ satisfy \eqref{eq:abscompl} with $w=w_j$ for all $j=n,n+1, \dots, s$.
Further, let $w_i \in V$ be given by \eqref{eq:defw}
% \begin{equation}
% \label{eq:defwi}
%   a(w_i,v) = a(u_{*,i},v) - \lambda_{*,i} b(u_{*,i},v) \quad\forall v \in V
% \end{equation}
and let $\eta_i > 0$ bounds $\norm{w_i}_a$ for all $i=r,r+1, \dots, s$ as in \eqref{eq:abscompl}.
Then $L_n$ defined in \eqref{eq:Kato} satisfies
\begin{equation}
\label{eq:Katobound}
%  \lambda_i &\leq \lambda_{*,i}, \\
%  \lambda_{*,n} \left( 1 + \nu \lambda_{*,n} \sum_{i=n}^s \frac{\eta_i^2}{\lambda_{*,i}^2(\nu - \lambda_{*,i})} \right)^{-1} \leq \lambda_n
  L_n \leq \lambda_n
\end{equation}
for all $n=r,r+1,\dots, s$.
%$\cred{\min\{s, s - r + 2\}}$.
%\cgreen{ALTERNATIVA 3: for all $n=r, r+1,\dots, s$ such that $\lambda_{n-1} \leq \nu$.}
\end{theorem}
\begin{proof}
Since \eqref{eq:smallEP} corresponds to a generalized matrix eigenvalue problem with a symmetric and positive definite matrix, eigenfunctions
%Let approximate eigenvalues $0 < \lambda_{*,r} \leq \lambda_{*,r+1} \leq \cdots \leq \lambda_{*,s}$ and the corresponding approximate eigenfunctions
$u_{*,r}, u_{*,r+1}, \dots, u_{*,s}$ form an orthogonal system:
\begin{equation}
  \label{eq:ONu*}
  a( u_{*,i}, u_{*,j} ) = \lambda_{*,i} \delta_{ij}
  \quad\text{and}\quad
  b( u_{*,i}, u_{*,j} ) = \delta_{ij}
  \quad \forall i,j=r,r+1,\dots,s.
\end{equation}

Let $n \in \{r, r+1, \dots, s\}$ %$\cred{\min\{s, s - r + 2\}}\}$ 
be arbitrary.
Let us consider a function $z_* = \sum_{i=n}^s \gamma_i u_{*,i}$, where coefficients $\gamma_i$ are uniquely determined by requirements $b(z_*, u_i) = 0$ for $i=n+1,n+2,\dots,s$ and $\sum_{i=n}^s \gamma_i^2 = 1$. Using \eqref{eq:ONu*}, we easily derive expressions
\begin{equation}
  \label{eq:normz}
  \norm{z_*}_a^2 = \sum_{i=n}^s \lambda_{*,i} \gamma_i^2
  \quad\text{and}\quad
  |z_*|_b^2 = 1.
\end{equation}
Now, we consider $w^z \in V$ given by
\begin{equation}
  \label{eq:defwz}
  a(w^z,v) = a(z_*,v) - \lambda_{*,n} b(z_*,v) \quad\forall v \in V
\end{equation}
and work out an expression for $\norm{w^z}_a$.
Using \eqref{eq:defwz}, definition of $z_*$, and \eqref{eq:defw}, we have
\begin{multline}
  \label{eq:anormwzB}
  \norm{w^z}_a^2 = a(z_*,w^z) - \lambda_{*,n} b(z_*,w^z)
   = \sum_{i=n}^s \gamma_i \left[ a(u_{*,i},w^z) - \lambda_{*,n} b(u_{*,i},w^z) \right]
\\
   = \sum_{i=n}^s \gamma_i \left[ a(w_i,w^z) + (\lambda_{*,i}- \lambda_{*,n}) b(u_{*,i},w^z) \right].
\end{multline}
Considering $i \in \{n,n+1,\dots,s\}$, we
notice that identities \eqref{eq:defw} and \eqref{eq:ONu*} easily imply $a(u_{*,j},w_i) = 0$ and $a(w_j,w_i) = -\lambda_{*,j} b(u_{*,j}, w_i)$ for all $j=n,n+1,\dots,s$.
Consequently, using \eqref{eq:defwz} with $v=w_i$ and definition of $z_*$, we derive equality
\begin{equation}
  \label{eq:awiwz}
  a(w_i,w^z) = \sum_{j=n}^s \gamma_j \left[ a(u_{*,j}, w_i) - \lambda_{*,n} b(u_{*,j}, w_i) \right]
   = \sum_{j=n}^s \gamma_j \frac{\lambda_{*,n}}{\lambda_{*,j}} a(w_j,w_i).
\end{equation}
Setting $v=u_{*,i}$ in \eqref{eq:defwz}, using definition of $z_*$ and orthogonality \eqref{eq:ONu*}, we have
\begin{equation}
  \label{eq:auiwz}
  a(u_{*,i}, w^z) = a(z_*, u_{*,i}) - \lambda_{*,n} b(z_*, u_{*,i}) = \gamma_i (\lambda_{*,i} - \lambda_{*,n}).
\end{equation}
The last auxiliary step is to take $v=w^z$ in \eqref{eq:defw} and utilize \eqref{eq:awiwz} and \eqref{eq:auiwz} to obtain
\begin{equation}
  \label{eq:buiwz}
  \lambda_{*,i} b(u_{*,i}, w^z) = a(u_{*,i}, w^z) - a(w_i, w^z)
   = \gamma_i (\lambda_{*,i} - \lambda_{*,n}) - \sum_{j=n}^s \gamma_j \frac{\lambda_{*,n}}{\lambda_{*,j}} a(w_j,w_i).
\end{equation}
Finally, we substitute \eqref{eq:awiwz} and \eqref{eq:buiwz} into \eqref{eq:anormwzB} and after straightforward manipulations we arrive at
\begin{equation}
  \label{eq:anormwz}
  \norm{w^z}_a^2 = \sum_{i=n}^s \gamma_i^2 \frac{(\lambda_{*,i} - \lambda_{*,n})^2}{\lambda_{*,i}}
    + \lambda_{*,n}^2 \sum_{i=n}^s \sum_{j=n}^s \frac{\gamma_i \gamma_j}{\lambda_{*,i}\lambda_{*,j}} a(w_i,w_j).
\end{equation}

The following chain of inequalities leads to the lower bound \eqref{eq:Katobound}.
%First notice that $(\lambda_i - \lambda_n)(\lambda_i - \nu) \geq 0$ for all $i=1,2,\dots$ due to the assumptions on $\nu$.
\cred{
%First notice that $\lambda_{*,s}$ is a Rayleigh--Ritz???Galerkin approximate eigenvalue computed in the $s-r+1$ dimensional space $\widetilde V_*$ and, thus, $\lambda_{s-r+1} \leq \lambda_{*,s}$. 
%Consequently, $\lambda_{n-1} \leq \lambda_{*,s}$  and 
Due to the \eqref{eq:nucond} we easily verify that $(\lambda_i - \lambda_n)(\lambda_i - \nu) \geq 0$ for all $i \leq n$ and all $i \geq s+1$. Combining this with the definition of $z_*$,
}
%Therefore, 
we have
$$
  0 \leq \sum_{i=1}^\Ninf \frac{1}{\lambda_i} (\lambda_i - \lambda_n)(\lambda_i - \nu) |b(z_*,u_i)|^2
  = \sum_{i=1}^\Ninf \left( \lambda_i - (\lambda_n+\nu) + \frac{\lambda_n \nu}{\lambda_i} \right) |b(z_*,u_i)|^2.
$$
Employing identities \eqref{eq:a-norm}, \eqref{eq:Parseval}, and \eqref{eq:bsqoverlambda} with $w=w^z$ defined by \eqref{eq:defwz}, $\lambda_* = \lambda_{*,n}$, and $u_* = z_*$, we obtain
$$
  0 \leq \norm{z_*}_a^2 - \norm{z_*^\cK}_a^2 - (\lambda_n + \nu) |z_*|_b^2
    + \frac{\lambda_n \nu}{\lambda_{*,n}^2} \left(
      \norm{w^z}_a^2 + 2\lambda_{*,n} |z_*|_b^2 - \norm{z_*}_a^2 \right).
$$
Now, we substitute relations \eqref{eq:normz} and \eqref{eq:anormwz} into this inequality, use that fact that $\norm{z_*^\cK}_a^2 \geq 0$ and derive
\begin{multline*}
  0 \leq
    \sum_{i=n}^s \lambda_{*,i} \gamma_i^2
    - \lambda_n - \nu
    + \frac{\lambda_n \nu}{\lambda_{*,n}^2} \left(
      \sum_{i=n}^s \gamma_i^2 \frac{(\lambda_{*,i} - \lambda_{*,n})^2}{\lambda_{*,i}}
\right. \\ \left.
    + \lambda_{*,n}^2 \sum_{i=n}^s \sum_{j=n}^s \frac{\gamma_i \gamma_j}{\lambda_{*,i}\lambda_{*,j}} a(w_i,w_j)
    + 2\lambda_{*,n}
    - \sum_{i=n}^s \lambda_{*,i} \gamma_i^2
    \right).
\end{multline*}
Using \cred{identity} $\sum_{i=n}^s \gamma_i^2 = 1$, we rearrange this inequality as
\begin{equation}
  \label{eq:aux1}
  \sum_{i=n}^s ( \nu - \lambda_{*,i}) \gamma_i^2
  \leq
  \frac{\lambda_n}{\lambda_{*,n}} \left[
  \sum_{i=n}^s \frac{\lambda_{*,n}}{\lambda_{*,i}} (\nu - \lambda_{*,i}) \gamma_i^2
  + \nu \lambda_{*,n} \sum_{i=n}^s \sum_{j=n}^s
    \frac{\gamma_i \gamma_j}{\lambda_{*,i}\lambda_{*,j}} a(w_i,w_j)
  \right].
\end{equation}
The first sum on the right hand side can be estimated as
\begin{equation}
  \label{eq:rhsest1}
  \sum_{i=n}^s \frac{\lambda_{*,n}}{\lambda_{*,i}} (\nu - \lambda_{*,i}) \gamma_i^2
   \leq \sum_{i=n}^s (\nu - \lambda_{*,i}) \gamma_i^2,
\end{equation}
because $\lambda_{*,n} \leq \lambda_{*,i} \cred{ < \nu}$ for all $i=n, n+1, \dots, s$.
The double sum on the right hand side of \eqref{eq:aux1} can be bounded using
$a(w_i,w_j) \leq \norm{w_i}_a \norm{w_j}_a \leq \eta_i \eta_j$ and the Cauchy--Schwarz inequality as
\begin{multline}
  \label{eq:rhsest2}
  \sum_{i=n}^s \sum_{j=n}^s
    \frac{\gamma_i \gamma_j}{\lambda_{*,i} \lambda_{*,j}} a(w_i,w_j)
  \leq
    \left( \sum_{i=n}^s |\gamma_i| \frac{\eta_i}{\lambda_{*,i}} \right)^2
\\    
  \leq \left( \sum_{i=n}^s (\nu - \lambda_{*,i}) \gamma_i^2 \right)
       \left( \sum_{i=n}^s \frac{\eta_i^2}{\lambda_{*,i}^2 (\nu - \lambda_{*,i})} \right).
\end{multline}
Estimating the right-hand side of \eqref{eq:aux1} by \eqref{eq:rhsest1} and \eqref{eq:rhsest2} and dividing by 
$\sum_{i=n}^s (\nu - \lambda_{*,i}) \gamma_i^2$, we end up with inequality
$$
  1 \leq \frac{\lambda_n}{\lambda_{*,n}} \left( 1 + \nu \lambda_{*,n} \sum_{i=n}^s \frac{\eta_i^2}{\lambda_{*,i}^2 (\nu - \lambda_{*,i})} \right),
$$
which is equivalent to \eqref{eq:Katobound}.
\end{proof}

Let us make a few remarks about the result presented in Theorem~\ref{th:Katobound}.
\cred{First, choosing $n=s$,} 
the bound \eqref{eq:Katobound} simplifies to
$$
  \lambda_{*,n} \left( 1 + \frac{\nu}{\lambda_{*,n}(\nu - \lambda_{*,n})} \eta_n^2 \right)^{-1} \leq \lambda_n.
$$
%and problem \eqref{eq:smallEP} reduces to a suitable choice of $\lambda_{*,n}$.
%the bound \eqref{eq:Kato} advertised in Section~\ref{se:intro}.
Second, accurate estimates $\eta_i$ of the representatives of residuals $\|w_i\|_a$ are computed by the flux reconstruction approach, which we describe in details below in Section~\ref{se:flux}.

Third, lower bound \eqref{eq:Katobound} depends on $\nu$, which is required to satisfy \eqref{eq:nucond}.
This condition cannot be guaranteed unless a lower bound on $\lambda_{s+1}$ is known. If an analytic (and perhaps rough) lower bound on $\lambda_{s+1}$ is known, we can use \eqref{eq:Katobound} to compute more accurate lower bounds on $\lambda_s$ and smaller eigenvalues. However, a priori known lower bounds on eigenvalues are rare in practice. In Section~\ref{se:numex} we illustrate how to compute the needed lower bound by a homotopy method described in \cite{Plum1990,Plum1991}.
% \cred{
% The advantage of the parameter $r$ is that its higher values allow for rougher lower bounds $\nu$ on $\lambda_{s+1}$.
% Indeed, we know that $\lambda_{s-r+1} \leq \lambda_{*,s}$ and these values are supposed to decrease with $r$. Thus, in view of assumption \eqref{eq:nucond} we obtain wider interval for $\nu$.
% }

A practical approach in applications, where the lower bounds need not to be guaranteed, is to set  $\nu = \ell_{s+1}$ computed by \eqref{eq:Wein}. The bound $\ell_{s+1}$ is guaranteed to be below $\lambda_{s+1}$ under the closeness condition \eqref{eq:closest}, but this condition cannot be verified unless lower bounds on the corresponding eigenvalues are already known.
However, resolving the eigenvalue problem with sufficient accuracy provides a good confidence that the closeness condition holds true. In Section~\ref{se:numex} we illustrate how to use the computed lower bounds to verify a posteriori if the closeness condition is likely to hold.
In addition, we mention that the closeness condition \eqref{eq:closest} is just a sufficient condition in Theorem~\ref{th:Weinlowerbound} and the bound \eqref{eq:Wein} can be below the exact eigenvalue even if the closeness condition is not satisfied. Numerical experiments we performed indicate that this is actually very common situation in practical computations, see Section~\ref{se:numex} below.

The fourth remark concerns the optimal use of Theorem~\ref{th:Katobound}. Estimate \eqref{eq:Katobound} provides $s-r+1$ lower bounds on eigenvalues $\lambda_r$, $\lambda_{r+1}$, \dots, $\lambda_s$, respectively. In particular, we can use the lower bound on $\lambda_s$ as a new value of $\nu$ and use Theorem~\ref{th:Katobound} again to obtain new lower bounds on $\lambda_r$, $\lambda_{r+1}$, \dots, $\lambda_{s-1}$. This process can be repeated $s-r$ times using for $\nu$ the best lower bound computed so far, see \cite{Plum1991}. As the final lower bounds we naturally choose the largest one computed. In Section~\ref{se:numex} below, we combine this recursive process with the lower bound \eqref{eq:Wein} to obtain sharper bounds on rough meshes.

Finally, we remark that Theorem~\ref{th:Katobound} requires the exact equality in \eqref{eq:smallEP}.
If we compute approximate eigenfunctions $\tilde u_{*,i}$ using standard approaches based on the finite element discretization, we can usually take $u_{*,i} = \tilde u_{*,i}$ and the error in identity \eqref{eq:smallEP} is on the level of machine precision.
If not, then we suggest to compute $u_{*,i}$ by solving the (small) eigenvalue problem \eqref{eq:smallEP} by a direct method.
Of course, the exact equality in \eqref{eq:smallEP} cannot be reached due to round-off errors, but this issue can be solved by the interval arithmetics as suggested in \cite{Plum1990,Plum1991}, for example.

\section{Lower bounds for symmetric elliptic operators}
\label{se:ellop}

This section introduces a symmetric elliptic eigenvalue problem for second-order partial differential operators with mixed boundary conditions and all necessary assumptions. It also briefly describes its finite element discretization.

We consider the following eigenvalue problem: find $\lambda_n > 0$ and $u_n \neq 0$ such that
\begin{alignat}{2}
\nonumber
  -\ddiv( \cA \nabla u_n ) + c u_n &= \lambda_n \beta_1 u_n &\quad &\text{in }\Omega, \\
\label{eq:EPstrong}
  (\cA \nabla u_n) \cdot \bn_\Omega + \alpha u_n &= \lambda_n \beta_2 u_n &\quad &\text{on }\GammaN, \\
\nonumber
  u_n &= 0 &\quad &\text{on }\GammaD.
\end{alignat}
In order to formulate this problem in a weak sense, we consider
$\Omega \subset \R^2$ to be a Lipschitz domain with boundary $\partial\Omega$ split into two relatively open disjoint parts $\GammaN$ and $\GammaD$. Symbol $\bn_\Omega$ stands for the unit outward normal vector. Diffusion matrix $\cA \in [L^\infty(\Omega)]^{2\times2}$, reaction coefficient $c \in L^\infty(\Omega)$ and coefficients $\beta_1 \in L^\infty(\Omega)$, $\alpha,\beta_2 \in L^\infty(\GammaN)$ are assumed to be piecewise constant.
Further, we assume $\beta_1 \geq 0$ and $\beta_2 \geq 0$.
Note that the assumption of piecewise constant coefficients and of the two-dimensionality of the domain are not essential. They are needed due to technical reasons connected with the flux reconstruction.
In order to guarantee the symmetry and ellipticity,
we assume $c \geq 0$ and $\alpha \geq 0$ and the matrix $\cA$ to be symmetric and uniformly positive definite,
i.e. we assume existence of a constant $C>0$ such that
$$
  \b{\xi}^T \cA(x) \b{\xi} \geq C |\b{\xi}|^2 \quad\forall \b{\xi} \in \R^2
  \text{ and for almost all } x\in\Omega,
$$
where $|\cdot|$ stands for the Euclidean norm.

Defining the usual space
\begin{equation}
  \label{eq:defV}
  V = \{ v \in H^1(\Omega) : v = 0 \text{ on } \GammaD \},
\end{equation}
we introduce the following weak formulation of \eqref{eq:EPstrong}: find $\lambda_n > 0$ and $u_n \in V \setminus \{0\}$ such that
\begin{equation}
\label{eq:EPweak}
  a(u_n,v) = \lambda_n b(u_n,v) \quad \forall v \in V,
\end{equation}
where
\begin{align}
\label{eq:blf}
  a(u,v) &= (\cA \nabla u,\nabla v) + (c u, v) + (\alpha u, v)_\GammaN,
%  \int_\Omega (\nabla u)^T \cA \nabla v \dx
% +\int_\Omega c u v \dx  + \int_\GammaN \alpha u v \dx[s].
\\
\label{eq:blfb}
  b(u,v) &= (\beta_1 u, v) + (\beta_2 u, v)_\GammaN,
\end{align}
$(\cdot,\cdot)$ stands for the $L^2(\Omega)$, and
$(\cdot,\cdot)_\GammaN$ for the $L^2(\GammaN)$ inner products.

We assume the form $a(u,v)$ to be a scalar product in $V$. This is the case if at least one of the following conditions is satisfied:
(a) $c > 0$ on a subset of $\Omega$ of positive measure,
(b) $\alpha > 0$ on a subset of $\GammaN$ of positive measure,
(c) measure of $\GammaD$ is positive.
In agreement with the notation introduced above, we denote by $\|{\cdot}\|_a$ and $|{\cdot}|_b$ the norm induced by $a({\cdot}, {\cdot})$ and the seminorm induced by $b({\cdot}, {\cdot})$, respectively.

The following theorem shows the validity of the crucial compactness assumption.
Consequently, eigenproblem \eqref{eq:EPweak} is well defined and posses all properties listed in Theorem~\ref{th:properties}.
\begin{theorem}
% \cred{
% THIS IS TRUE EVEN IF $\beta_1=\beta_2 = 0$!!! HOWEVER, IF at least one of the following conditions is satisfied: 
% (d) $\beta_1 > 0$ on a subset of $\Omega$ of positive measure,
% (e) $\beta_2 > 0$ on a subset of $\GammaN$ of positive measure,
% THEN THE SPECTRUM IS INFINITE.
% }
Let bilinear forms $a(\cdot,\cdot)$ and $b(\cdot,\cdot)$ be defined by \eqref{eq:blf} and \eqref{eq:blfb} with the above listed requirements on the coefficients. Let $a(\cdot,\cdot)$ be a scalar product in $V$.
%Then the solution operator $S$ defined by \eqref{eq:solop} is compact.
Then the seminorm $|\cdot|_b$ is compact with respect to the norm $\norm{\cdot}_a$.
\end{theorem}
\begin{proof}
Notice that the definition \eqref{eq:blfb} of the form $b(u,v)$ for $u,v \in V$ is understood as follows
$$
  b(u,v) = (\beta_1 I u, I v)_{L^2(\Omega)} + (\beta_2 \gamma u, \gamma v)_{L^2(\GammaN)},
$$
where $I : V \rightarrow L^2(\Omega)$ is the identity mapping and $\gamma : V \rightarrow L^2(\GammaN)$ the trace operator. The identity $I$ is compact due to Rellich theorem \cite[Theorem~6.3]{Adams_Sob_spaces_03} and the compactness of the trace operator $\gamma$ is proved in \cite[Theorem 6.10.5]{Kuf_John_Fucik_Function_spaces}; see also \cite{Biegert:2009}.

Compactness of $I$ and $\gamma$ implies that from any sequence $\{v_n\} \subset V$ bounded in $\norm{\cdot}_a$ we can extract a subsequence $\{v_{n_k}\}$ such that $\{Iv_{n_k}\}$ is Cauchy in $L^2(\Omega)$ and $\{\gamma v_{n_k}\}$ is Cauchy in $L^2(\GammaN)$. Since
$$
  | v |_b^2 \leq \max\left\{ \norm{\beta_1}_{L^\infty(\Omega)}, \norm{\beta_2}_{L^\infty(\GammaN)} \right\}
    \left( \norm{ I v }_{L^2(\Omega)}^2 + \norm{ \gamma v}_{L^2(\GammaN)}^2 \right)
    \quad \forall v \in V,
$$
we immediately see that the subsequence $\{ v_{n_k} \}$ is Cauchy in $| \cdot |_b$ as well.
\end{proof}

We discretize the eigenvalue problem \eqref{eq:EPweak} by the standard conforming finite element method.
To avoid technicalities with curved elements, we assume the domain $\Omega$ to be polygonal
and consider a conforming (face-to-face) triangular mesh $\cT_h$ consisting of closed triangles called elements. Further, we define the finite element space
\begin{equation}
\label{eq:Vh}
  V_h = \{ v_h \in V : v_h|_K \in P_p(K),\ \forall K \in \cT_h\},
\end{equation}
where $P_p(K)$ stands for the space of polynomials of degree at most $p$ on the triangle $K \in \cT_h$.
With this notation, the finite element approximation of the eigenvalue problem~\eqref{eq:EPweak} reads:
Find $\lambda_{h,n} > 0$ and $u_{h,n} \in V_h\setminus\{0\}$ such that
\begin{equation}
\label{eq:EPFEM}
  a(u_{h,n}, v_h) = \lambda_{h,n} b(u_{h,n}, v_h ) \quad \forall v_h \in V_h.
\end{equation}

\section{Guaranteed bounds based on the complementary energy}
\label{se:compl}

Lower bounds $\ell_n$ and $L_n$ given by \eqref{eq:Wein} and \eqref{eq:Kato}, respectively, require a computable guaranteed upper bounds $\eta_n$ on the energy norm $\| w_n \|_a$ of the representative of the residual defined in \eqref{eq:defw}. In this section, we show how to compute this $\eta_n$. The technique described here is based on the complementary energy (or two energy principle), which can be traced back to the hypercircle method \cite{Synge:1957}, see also \cite{AinOde:2000,Braess2013,Rep:2008}.
In this section, we provide an estimate that is close to optimal with respect to the data computed solely on individual elements. We achieve this by combining several possibilities how to estimate various terms involved.

Let $V$ be given by \eqref{eq:defV} and let $\lambda_{*,n} > 0$ and $u_{*,n} \in V\setminus\{0\}$ be an arbitrary approximation of an eigenpair of \eqref{eq:EPstrong}. Further, let $\bq \in \Hdiv$ be an arbitrary vector field. Let us note that the quality of the resulting error bound depends heavily on an appropriate choice of $\bq$. A particular way how to compute suitable $\bq$ efficiently is described below in Section~\ref{se:flux}. However, the result presented in this section is valid for $\bq \in \Hdiv$ independently on the way it is constructed.

First, we introduce certain notation. Based on $\lambda_{*,n}$, $u_{*,n}$, $\bq$, we define flux $F$, residual $r$ in $\Omega$, and boundary residual $g$ on $\GammaN$ as
\begin{align}
  \nonumber
  F &= \cA \nabla u_{*,n} - \bq,
\\ \label{eq:Frg}  
  r &= c u_{*,n} - \lambda_{*,n} \beta_1 u_{*,n}  - \ddiv \bq,
  % \text{ in }\Omega,\quad\text{and}\quad
\\ \nonumber
  g &= \alpha u_{*,n} - \lambda_{*,n} \beta_2 u_{*,n} + \bq \cdot \bn_\Omega.
  %\text{ on }\GammaN.
\end{align}
Then, we compute their various norms for every element $K \in\cT_h$:
\begin{align}
  \label{eq:defFK}
  F_K &= \norm{\cA_K^{-1/2} F}_K, %= (\cA^{-1} F, F)_K^{1/2} = \norm{\cA^{-1} F}_{\cA,K},
\\
  \label{eq:defr1K}
  r_{1,K} &= \left\{ \begin{array}{cl}
     \norm{c_K^{-1/2} r}_K & \text{ if } c_K > 0, \\
     \infty & \text{ if } c_K = 0,
     \end{array} \right.
\\
  \label{eq:defr2K}
  r_{2,K} &= \left\{ \begin{array}{cl}
     h_K\pi^{-1}(\lambda_{\cA_K}^\mathrm{min})^{-1/2} \norm{r}_K & \text{ if } \int_K r \dx = 0, \\
     \infty & \text{ otherwise, }
     \end{array} \right.
\\
  \label{eq:defr3K}
  r_{3,K} &= \left\{ \begin{array}{cl}
     \norm{\beta_{1,K}^{-1/2} r}_K & \text{ if } \beta_{1,K} > 0, \\
     \infty & \text{ if } \beta_{1,K} = 0.
     \end{array} \right.
\end{align}
Here, quantities $\cA_K$, $c_K$, and $\beta_{1,K}$ are constant values of coefficients $\cA$, $c$, and $\beta_1$ restricted to the element $K$, respectively.
Symbol $\norm{\cdot}_K$ stands for the usual norm in $L^2(K)$. In general, we adopt the notation $\norm{\cdot}_Q$ and $(\cdot,\cdot)_Q$ for the norm and the inner product in $L^2(Q)$, where $Q$ is a domain.
We also denote by
$\norm{\cdot}_{a,K}$ and $|\cdot|_{b,K}$ the local energy norm and the local $b$-seminorm, i.e.
$\norm{v}_{a,K}^2=(\cA_K \nabla v,\nabla v)_K + (c_K v, v)_K + (\alpha v, v)_{\partial K \cap\GammaN}$
and $|v|_{b,K}^2 = (\beta_{1,K} v, v)_K + (\beta_2 v, v)_{\partial K \cap\GammaN}$,
cf. \eqref{eq:blf} and \eqref{eq:blfb}.
We recall that $\lambda_{\cA_K}^\mathrm{min}$ stands for the smallest eigenvalue of the local coefficient matrix $\cA_K$ and $h_K$ for the diameter of the element $K$.

Further, we define similar quantities on those edges of the element $K \in \cT_h$ that lie on $\GammaN$.
We put $\cEN_K = \{ \gamma: \gamma \text{ is an edge of } K \text{ and } \gamma\subset\partial K \cap \GammaN \}$
and for all elements $K\in\cT_K$ and all edges $\gamma \in \cEN_K$ we define
\begin{align}
  \label{eq:defg1gamma}
  g_{1,\gamma} &= \left\{ \begin{array}{cl}
     \norm{\alpha_\gamma^{-1/2} g}_\gamma & \text{ if } \alpha_\gamma > 0, \\
     \infty & \text{ if } \alpha_\gamma = 0,
     \end{array} \right.
\\
  \label{eq:defg2gamma}
  g_{2,\gamma} &= \left\{ \begin{array}{cl}
     \min\left\{ C_K^{\cA_K}, \overline{C}_K^{\cA_K} \right\} \norm{g}_\gamma, & \text{ if } c_K>0 \text{ or } \int_\gamma g \dx = 0, \\
     \infty & \text{ otherwise},
     \end{array} \right.
\\
  \label{eq:defg3gamma}
  g_{3,\gamma} &= \left\{ \begin{array}{cl}
     \norm{\beta_{2,\gamma}^{-1/2} g}_\gamma & \text{ if } \beta_{2,\gamma} > 0, \\
     \infty & \text{ if } \beta_{2,\gamma} = 0.
     \end{array} \right.
\end{align}
Similarly as above, $\alpha_\gamma$ and $\beta_{2,\gamma}$ are the constant values of $\alpha$ and $\beta_2$ on edges $\gamma \in \cEN_K$.
Constants $C_K^{\cA_K}$ and $\overline{C}_K^{\cA_K}$ are given by simple modifications of \cite[Lemma 1]{AinVej2014corr,AinVej2014}
as
\begin{align*}
  \left(C_K^{\cA_K}\right)^2 &= \frac{|\gamma|}{d|K| c_K^{1/2}}
    \left( \frac{4h_K^2}{\lambda_{\cA_K}^\mathrm{min}} + \frac{d^2}{c_K} \right)^{1/2},
\\
%   \left(\overline{C}_K^{\cA_K}\right)^2  &= \frac{|\gamma|}{d|K|}
%     \min\left\{ \frac{h_K}{ \pi \left(\lambda_{\cA_K}^\mathrm{min}\right)^{1/2} }, \frac{1}{c_K^{1/2}} \right\}
%     \left( \frac{2h_K}{\left(\lambda_{\cA_K}^\mathrm{min}\right)^{1/2}}
%        + d \min\left\{ \frac{h_K}{ \pi \left(\lambda_{\cA_K}^\mathrm{min}\right)^{1/2} }, \frac{1}{c_K^{1/2}} \right\} \right),
%
  \left(\overline{C}_K^{\cA_K}\right)^2  &= \frac{|\gamma|}{d|K|}
    \overline{M}_K^{\cA_K}
    \left( \frac{2h_K}{\left(\lambda_{\cA_K}^\mathrm{min}\right)^{1/2}}
       + d\,\overline{M}_K^{\cA_K} \right),
%   \quad\text{where }
%   \overline{M}_K^{\cA_K} = \min\left\{ \frac{h_K}{ \pi \left(\lambda_{\cA_K}^\mathrm{min}\right)^{1/2} }, \frac{1}{c_K^{1/2}} \right\}
\end{align*}
where $\overline{M}_K^{\cA_K} = \min\left\{ h_K \pi^{-1} \left(\lambda_{\cA_K}^\mathrm{min}\right)^{-1/2},  c_K^{-1/2} \right\}$
and $d=2$ stands for the dimension.
Note that we formally consider the possibility $C_K^{\cA_K} = \infty$ in the case of $c_K=0$
and $\overline{C}_K^{\cA_K} = \infty$ if $\int_\gamma g \dx \neq 0$.

To proceed, we introduce sets
$\cENO_K = \{ \gamma \in \cEN_K : \beta_{2,\gamma} = 0 \}$ and
$\cENP_K= \{ \gamma \in \cEN_K : \beta_{2,\gamma} > 0 \}$.
For all elements $K\in \cT_h$ we define
\begin{multline}
  \label{eq:defM}
  \MM_K = \min\left\{
    \left( F_K^2 + r_{1,K}^2 + \sum_{\gamma\in\cEN_K} g_{1,\gamma}^2 \right)^{1/2},\
    \left( F_K^2 + r_{1,K}^2 \right)^{1/2} + \sum_{\gamma\in\cEN_K} g_{2,\gamma},
  \right. \\ \left.
    \left( F_K^2 + \sum_{\gamma\in\cEN_K} g_{1,\gamma}^2 \right)^{1/2} + r_{2,K},\
    F_K + r_{2,K} + \sum_{\gamma\in\cEN_K} g_{2,\gamma}
  \right\}
\end{multline}
and also quantity $\MM_K^0$, which is given by \eqref{eq:defM} with $\cEN_K$ replaced by $\cENO_K$.
It is also useful to put
\begin{align}
  \label{eq:defLK}
  Y_K &= \min \left\{ \left( F_K^2 + \sum_{\gamma\in\cENO_K} g_{1,\gamma}^2\right)^{1/2},\
         F_K + \sum_{\gamma\in\cENO_K} g_{2,\gamma} \right\},
\\
  \label{eq:defRK}
  R_K &= \left( r_{3,K}^2 + \sum_{\gamma\in\cENP_K} g_{3,\gamma}^2 \right)^{1/2},
\\
  \label{eq:defGK}
  G_K &= \left( \sum_{\gamma \in \cENP_K} g_{3,\gamma}^2 \right)^{1/2}.
\end{align}

Finally, if a quantity $\underline{\lambda}_1$ such that $0 < \underline{\lambda}_1 \leq \lambda_1$ is available, we put
\begin{align*}
  \cT_h^{++} &= \{ K \in \cT_h: \beta_{1,K}>0 \text{ and } Y_K + \underline{\lambda}_1^{-1/2} R_K \leq \MM_K \},
\\
  \cT_h^{+0} &=
  \{ K \in \cT_h: \beta_{1,K}=0,\ \cENP_K \neq \emptyset,
  \text{ and } \underline{\lambda}_1^{-1/2} G_K + \MM_K^0 \leq \MM_K \},
\end{align*}
$\cT_h^+ = \cT_h^{++} \cup \cT_h^{+0}$ and $\cT_h^0 = \cT_h \setminus \cT_h^+$.
Note that if lower bound $\underline{\lambda}_1$ is not available, we formally set $\cT_h^{++} = \cT_h^{+0} = \emptyset$.
% Further, $\cT_h^{++} = \{ K \in \cT_h^+ : \beta_{1,K} > 0 \}$
% and $\cT_h^{+0} = \{ K \in \cT_h^+ : \beta_{1,K} = 0 \}$.
Using this notation, we formulate the following theorem.
%-------------------------------------
\begin{theorem}
\label{th:estimator}
Let $\lambda_{*,n} \in \R$ and $u_{*,n} \in V \setminus \{0\}$ be arbitrary and
let $w_n\in V$ be given by \eqref{eq:defw}. Using the above notation, we have
\begin{equation}
\label{eq:eta}
  \norm{w_n}_a \leq \eta_n = \min \{ \eta_n^{(\mathrm{a})}, \eta_n^{(\mathrm{b})} \}
\end{equation}
where
\begin{multline*}
%   \eta = \min\left\{ \left( \sum_{K\in\cT_h} \MM_K^2 \right)^{1/2},
%      \underline{\lambda}_1^{-1/2} \left( \sum_{K\in\cT_h^+} \NN_K^2 \right)^{1/2}
%    + \left( \sum_{K\in\cT_h^0} \MM_K^2 + \sum_{K\in\cT_h^+} F_K^2 \right)^{1/2}
%   \right\}.
  \eta_n^{(\mathrm{a})} = \left( \sum_{K\in\cT_h} \MM_K^2 \right)^{1/2}, \quad
  \eta_n^{(\mathrm{b})} = \left[
  \sum_{K \in \cT_h^{++}} Y_K^2 + \sum_{K \in \cT_h^{+0}} (M_K^0)^2 + \sum_{K \in \cT_h^0} M_K^2
  \right]^{1/2}
\\  
  +
  \underline{\lambda}_1^{-1/2} \left[ \sum_{K \in \cT_h^{++}} R_K^2 + \sum_{K \in \cT_h^{+0}} G_K^2 \right]^{1/2}.
\end{multline*}
\end{theorem}
%-------------------------------------
\begin{proof}
For brevity, we will write $w=w_n$ within this proof.
First, we prove the estimate $\norm{w}_a \leq \eta_n^{(\mathrm{a})}$.
Using \eqref{eq:defw*} and the divergence theorem, we easily obtain identity
\begin{equation}
  \label{eq:enormw}
  \norm{w}_a^2 = (F,\nabla w) + (r,w) + (g,w)_\GammaN
  = \sum_{K\in\cT_h} \cL_K,
\end{equation}
where $\cL_K = (F,\nabla w)_K + (r,w)_K + \sum_{\gamma\in\cEN_K} (g,w)_\gamma$.
The terms in $\cL_K$ can be estimated by the Cauchy--Schwarz inequality. We bound the first one as
\begin{equation}
  \label{eq:FwK}
  (F,\nabla w)_K =  (\cA^{-1/2} F, \cA^{1/2} \nabla w ) \leq \norm{\cA^{-1/2}F}_K \norm{\cA^{1/2}\nabla w}_K
    = F_K \norm{\nabla w}_{\cA,K},
\end{equation}
where we use the notation $\norm{\nabla w}_{\cA,K}^2 = (\cA \nabla w, \nabla w)_K$.
The second term can be bounded in three ways.
First, in those elements $K\in\cT_h$ where $c_K > 0$, we have
$(r,w)_K %= (c_K^{-1/2} r, c_K^{1/2} w)_K
\leq c_K^{-1/2}\norm{r}_K \norm{ c_K^{1/2} w}_K$.
Second, in elements $K\in\cT_h$ where $\int_K r \dx = 0$, we can consider $\bar w_K = |K|^{-1} \int_K w \dx$,
Poincar\'e inequality $\norm{w - \bar w_K}_K \leq h_K \pi^{-1} \norm{\nabla w}_K$ \cite{Payne_Weinberger_opti_Poinc_ineq_conv_domains_1960}, the smallest eigenvalue
$\lambda_{\cA_K}^\mathrm{min}$ of $\cA_K$,
and derive estimate
$(r,w)_K = (r, w - \bar w_K)_K \leq \norm{r}_K \norm{w - \bar w_K}_K \leq h_K \pi^{-1} (\lambda_{\cA_K}^\mathrm{min})^{-1/2} \norm{r}_K \norm{\nabla w}_{\cA,K}$.
Third, in elements $K\in\cT_h$, where $\beta_{1,K} > 0$, we have
$(r,w)_K %= (\beta_{1,K} r, \beta_{1,K}^{1/2} w)_K
\leq \beta_{1,K}^{-1/2} \norm{r}_K \norm{ \beta_{1,K}^{1/2} w}_K$.
Thus, using definitions \eqref{eq:defr1K}--\eqref{eq:defr3K}, we obtain estimates
\begin{align}
  \label{eq:rwK1}
  (r,w)_K &\leq r_{1,K} \norm{ c_K^{1/2} w}_K, \\
  \label{eq:rwK2}
  (r,w)_K &\leq r_{2,K} \norm{\nabla w}_{\cA,K}, \\
  \label{eq:rwK3}
  (r,w)_K &\leq r_{3,K} \norm{ \beta_{1,K}^{1/2} w}_K
\end{align}
for all elements $K\in\cT_h$.

In a similar way, we estimate the third term.
First, on edges $\gamma \in \cEN_K$ where $\alpha_\gamma > 0$, we have
$(g,w)_\gamma %= (\alpha_\gamma^{-1/2} g, \alpha_\gamma^{1/2} w)_\gamma
\leq \alpha_\gamma^{-1/2} \norm{g}_\gamma \norm{\alpha_\gamma^{1/2} w}_\gamma$.
Second, on those edges $\gamma \in \cEN_K$ where $\int_\gamma g \dx = 0$, we can use the trace inequality \cite[Lemma 1]{AinVej2014corr} and obtain
$(g,w)_\gamma = (g, w- \bar w_\gamma)_\gamma \leq \norm{g}_\gamma \norm{w-\bar w_\gamma}_\gamma
               \leq \overline{C}_K^{\cA_K} \norm{g}_\gamma \norm{w}_{a,K}$,
where $\bar w_\gamma = |\gamma|^{-1} \int_\gamma w \dx$.
Third, if $c_K > 0$ and $\gamma \in \cEN_K$ then the trace inequality \cite[Lemma 1]{AinVej2014corr}
yields
$(g,w)_\gamma = \norm{g}_\gamma \norm{w}_\gamma \leq C_K^{\cA_K} \norm{g}_\gamma \norm{w}_{a,K}$.
Fourth, on edges $\gamma \in \cEN_K$ where $\beta_{2,\gamma} > 0$, we have
$(g,w)_\gamma %= (\beta_{2,\gamma}^{-1/2} g, \beta_{2,\gamma}^{1/2} w)_\gamma
               \leq \beta_{2,\gamma}^{-1/2} \norm{g}_\gamma \norm{\beta_{2,\gamma}^{1/2} w}_\gamma$.
Thus, using definitions \eqref{eq:defg1gamma}--\eqref{eq:defg3gamma} we obtain the following bounds
%TO DO: PUT IT INTO ONE LINE???
\begin{align}
  \label{eq:gwgamma1}
  (g,w)_\gamma &\leq g_{1,\gamma} \norm{\alpha_\gamma^{1/2} w}_\gamma, \\
  \label{eq:gwgamma2}
  (g,w)_\gamma &\leq g_{2,\gamma} \norm{w}_{a,K}, \\
  \label{eq:gwgamma3}
  (g,w)_\gamma &\leq g_{3,\gamma} \norm{\beta_{2,\gamma}^{1/2} w}_\gamma.
\end{align}

Now, we bound $\cL_K$ in four ways using various combinations of bounds \eqref{eq:FwK}, \eqref{eq:rwK1}--\eqref{eq:rwK2}, and \eqref{eq:gwgamma1}--\eqref{eq:gwgamma2}:
\begin{align*}
  \cL_K &\leq F_K \norm{\nabla w}_{\cA,K} + r_{1,K} \norm{ c_K^{1/2} w}_K
                   + \sum_{\gamma\in\cEN_K} g_{1,\gamma} \norm{\alpha_\gamma^{1/2} w}_\gamma,
% \\                   
%    &\leq \left( F_K^2 + r_{1,K}^2 + \sum_{\gamma\in\cEN_K} g_{1,\gamma}^2 \right)^{1/2} \norm{w}_{a,K},
\\
  \cL_K &\leq F_K \norm{\nabla w}_{\cA,K} + r_{1,K} \norm{ c_K^{1/2} w}_K
                   + \sum_{\gamma\in\cEN_K} g_{2,\gamma} \norm{w}_{a,K},
% \\                   
%    &\leq \left[ \left( F_K^2 + r_{1,K}^2\right)^{1/2} + \sum_{\gamma\in\cEN_K} g_{2,\gamma} \right] \norm{w}_{a,K},
\\
  \cL_K &\leq F_K \norm{\nabla w}_{\cA,K} + r_{2,K} \norm{\nabla w}_{\cA,K}
                   + \sum_{\gamma\in\cEN_K} g_{1,\gamma} \norm{\alpha_\gamma^{1/2} w}_\gamma,
% \\                   
%    &\leq \left[ \left( F_K^2 + \sum_{\gamma\in\cEN_K} g_{1,\gamma}^2 \right)^{1/2} + r_{2,K}\right] \norm{w}_{a,K},
\\
  \cL_K &\leq F_K \norm{\nabla w}_{\cA,K} + r_{2,K} \norm{\nabla w}_{\cA,K}
                   + \sum_{\gamma\in\cEN_K} g_{2,\gamma} \norm{w}_{a,K}.
% \\                   
%    &\leq \left( F_K + r_{2,K} + \sum_{\gamma\in\cEN_K} g_{2,\gamma} \right) \norm{w}_{a,K}.
\end{align*}
Using the Cauchy-Schwarz inequality in these estimates, we conclude that
\begin{equation}
  \label{eq:MKest}
  %(F,\nabla w)_K + (r,w)_K + \sum_{\gamma\in\cEN_K} (g,w)_\gamma \leq M_K \norm{w}_{a,K}
  \cL_K \leq M_K \norm{w}_{a,K}
\end{equation}
and consequently equality \eqref{eq:enormw} yields
$$
  \norm{w}_a^2 \leq \sum_{K\in\cT_h} M_K \norm{w}_{a,K} \leq \left( \sum_{K\in\cT_h} M_K^2 \right)^{1/2} \norm{w}_a,
$$
which readily provides the desired bound $\norm{w}_a \leq \eta_n^{(\mathrm{a})}$.

To prove $\norm{w}_a \leq \eta_n^{(\mathrm{b})}$, we first derive several auxiliary estimates using the same technique as above.
First, using \eqref{eq:defFK}, \eqref{eq:defg1gamma}, and \eqref{eq:defg2gamma}, we obtain
\begin{equation}
  \label{eq:LK}
  (F,\nabla w)_K + \sum_{\gamma \in \cENO_K} (g,w)_\gamma \leq Y_K \norm{w}_{a,K},
\end{equation}
where $Y_K$ is given by \eqref{eq:defLK}.
Similarly, by \eqref{eq:defr3K} and \eqref{eq:defg3gamma}, we have
\begin{equation}
  \label{eq:RK}
  (r,w)_K + \sum_{\gamma\in\cENP_K} (g,w)_\gamma \leq R_K |w|_{b,K},
\end{equation}
where $R_K$ is provided in \eqref{eq:defRK}.
Next, using the same approach as for \eqref{eq:MKest}, we derive the bound
\begin{equation}
  \label{eq:M0K}
  (F,\nabla w)_K + (r,w)_K + \sum_{\gamma \in \cENO_K} (g,w)_\gamma \leq \MM_K^0 \norm{w}_{a,K},
\end{equation}
where $\MM_K^0$ was introduced below \eqref{eq:defM}.
% $\MM_K^0 = \min\left\{
%     \left( F_K^2 + r_{1,K}^2 + \sum_{\gamma\in\cENO_K} g_{1,\gamma}^2 \right)^{\frac12},
%     \left( F_K^2 + r_{1,K}^2 \right)^{\frac12} + \sum_{\gamma\in\cENO_K} g_{2,\gamma},
%   \right. \\ \left.
%     \left( F_K^2 + \sum_{\gamma\in\cENO_K} g_{1,\gamma}^2 \right)^{\frac12} + r_{2,K},
%     F_K + r_{2,K} + \sum_{\gamma\in\cENO_K} g_{2,\gamma}
%   \right\}$
%is the same as $\MM_K$ given by \eqref{eq:defM} except that here we sum edges over $\cENO_K$.
Finally, we use \eqref{eq:defg3gamma} and estimate
\begin{equation}
  \label{eq:GK}
  \sum_{\gamma \in \cENP_K} (g,w)_\gamma \leq G_K |w|_{b,K},
\end{equation}
where $G_K$ is defined in \eqref{eq:defGK}.

Thus, using estimates \eqref{eq:LK}, \eqref{eq:RK}, \eqref{eq:M0K}, and \eqref{eq:GK} in \eqref{eq:enormw}, we obtain
\begin{multline*}
\norm{w}_a^2 =
%   \sum_{K \in \cT_h^{++}} \left[
%     (F,\nabla w)_K + (r,w)_K + \sum_{\gamma\in\cENP_K} (g,w)_\gamma + \sum_{\gamma\in\cENO_K} (g,w)_\gamma \right] \\
% &\quad
%   + \sum_{K \in \cT_h^{+0}} \left[
%     (F,\nabla w)_K + (r,w)_K + \sum_{\gamma\in\cENP_K} (g,w)_\gamma + \sum_{\gamma\in\cENO_K} (g,w)_\gamma \right] \\
% &\quad
%   + \sum_{K \in \cT_h^0} \left[
%     (F,\nabla w)_K + (r,w)_K + \sum_{\gamma\in\cEN_K} (g,w)_\gamma \right] \\
  \sum_{K \in \cT_h^{++}} \cL_K
  + \sum_{K \in \cT_h^{+0}} \cL_K
  + \sum_{K \in \cT_h^0} \cL_K
\leq
  \sum_{K \in \cT_h^{++}} \left[ Y_K \norm{w}_{a,K} 
\right. \\ \left.
    + R_K |w|_{b,K} \right]
  + \sum_{K \in \cT_h^{+0}} \left[ M_K^0 \norm{w}_{a,K} + G_K |w|_{b,K} \right]
  + \sum_{K \in \cT_h^0} M_K \norm{w}_{a,K}.
\end{multline*}
This can be further estimated using the Cauchy--Schwarz inequality as
\begin{multline*}
\norm{w}_a^2 \leq \left[
  \sum_{K \in \cT_h^{++}} Y_K^2 + \sum_{K \in \cT_h^{+0}} (M_K^0)^2 + \sum_{K \in \cT_h^0} M_K^2
  \right]^{1/2} \norm{w}_a
\\
  +
  \left[ \sum_{K \in \cT_h^{++}} R_K^2 + \sum_{K \in \cT_h^{+0}} G_K^2 \right]^{1/2} |w|_b
\end{multline*}
Since $ |w|_b \leq \lambda_1^{-1/2} \norm{w}_a \leq \underline{\lambda}_1^{-1/2} \norm{w}_a$,
the proof is finished.
\end{proof}

\section{Local flux reconstruction}
\label{se:flux}

In this section, we describe a local procedure how to construct a suitable flux reconstruction $\bq \in \Hdiv$ needed to evaluate $\eta_n$ in Theorem~\ref{th:estimator}.
Specifically, we use the flux reconstruction proposed in \cite{BraSch:2008} and generalize it to the eigenvalue problem \eqref{eq:EPweak}.
The description of the local flux reconstruction is technical and it is inspired mainly by works \cite{Ern_Voh_adpt_IN_13} and \cite{Dol_Ern_Vohr_hp_refinement_strategies_polyn_rob_AEE}.
We will denote the computable flux reconstruction by $\tqT$ in order to distinguish it from an arbitrary element $\bq \in \Hdiv$.
The flux reconstruction $\tqT$ is computed as a reconstruction of the already computed approximate gradient $\nabla u_{h,n}$, where $u_{h,n} \in V\setminus\{0\}$ is given by \eqref{eq:EPFEM}. %However, no specific assumptions about $\lambda_*$ and $u_*$ are needed and we consider them to be arbitrary in this section.

The flux reconstruction $\tqT$ is naturally defined in the Raviart--Thomas finite element spaces, see e.g.~\cite{BreFor:1991} and \cite{Quar_Val_Num_appr_PDE_94}.
%Let $\tRTNpK := [P_{p}(K)]^d + \boldsymbol{x}P_{p}(K)$ for $K \in \cT_h$.
%Let $\tilde{P_p}(K)$ denote the space of homogeneous polynomials of degree $p$ on $K \in \cT_h$.
The Raviart--Thomas space of order $p$ is defined on the mesh $\cT_h$ as
\begin{equation}\label{RTN}
\bWT = \left\{  \bwT \in \Hdiv:
    \bwT|_K \in \RT(K) \quad \forall K \in \cT_h
  \right \},
\end{equation}
where $\RT(K) = [P_p(K)]^2 \oplus \bfx P_p(K)$ and $\bfx = (x_1,x_2)$ is the vector of coordinates.
% ??? DO WE NEED THIS: ???
% Functions $\bwT \in \bWT$ have continuous normal components across the internal edges
% and $\ddiv \bwT|_K \in P_p(K)$ for all $K \in \cT_h$.
% In addition, the normal components of $\bwT$ on edges $\Gamma$ span the whole space $P_p(\Gamma)$ of polynomials of degree at most $p$ on $\Gamma$ and we have
% \begin{equation}
% \label{eq:normcomp=Pp}
% \left\{ \bwT|_\Gamma {\cdot} \bn_\Gamma : \bwT \in \bWT \right\} = P_{p}(\Gamma)
% \quad \forall \Gamma \in \cET.
% \end{equation}

We introduce the notation for vertices (nodes) of the mesh $\cT_h$.
Let $\NT$ denote the set of all vertices in $\cT_h$. The subsets of those lying on $\oGammaD$, on $\GammaN$, and in the interior of $\Omega$ are denoted by $\NTD$, $\NTN$, and $\NTI$, respectively. Notice that if a vertex is located at the interface between the Dirichlet and Neumann boundary, it is not in $\NTN$, but only in $\NTD$.
We also denote by $\NTK$ and $\cETK$ the sets of three vertices and three edges of the element $K$, respectively.

We construct the flux reconstruction $\tqT \in \bWT$ by solving local Neumann and Neumann/Dirichlet mixed finite element problems defined on patches of elements sharing a given vertex.
Let $\ba \in \NT$ be an arbitrary vertex, we denote by $\HF$ the standard piecewise linear and continuous hat function associated with $\ba$. This function vanishes at all vertices of $\cT_h$ except of $\ba$, where it has value 1. Note that $\HF \in \VT$ for vertices $\ba \in \NTI \cup \NTN$, but $\HF \not\in \VT$ for $\ba \in \NTD$.
Further, let $\Ta = \{ K \in \cT_h : \ba \in K \}$ be the set of elements sharing the vertex~$\ba$ and
$\Oma = \interior\bigcup \{ K: K \in \Ta \}$ the patch of elements sharing the vertex $\ba$.
%, and let $\poma$ be its boundary.
We denote by $\cEaI$ the set of interior edges in the patch $\Oma$,
by $\cEaBE$ the set of those edges on the boundary $\poma$ that do not contain $\ba$,
and by $\cEaBD$ and $\cEaBN$ the sets of edges on the boundary $\poma$ with an end point at $\ba$ lying either on $\GammaD$ or on $\GammaN$, respectively. Note that sets $\cEaBD$ and $\cEaBN$ can be nonempty only if $\ba \in \NTD \cup \NTN$, i.e. for boundary patches.
% let $\poma$ be its boundary, and $\cEaI$ the set of its interior edges.
% If $\ba \in \NTI$ then we set $\cEaB$ to be the sets of its edges on the boundary $\poma$.
% However, for $\ba \in \NTD\cup\NTN$ we denote by $\cEaB$ only those edges on $\poma$ that do not contain the vertes $\ba$. The set edges lying on the boundary $\poma$, containing vertex $\ba$, and being on $\GammaD$ is denoted by $\FaD$, while the set of those lying on $\GammaN$ by $\FaN$.

% In this section, we consider $\lambda_{h,n}\in\R$ and $u_{h,n}\in\VT$ to be a fixed and arbitrary approximation of an eigenpair of \eqref{eq:EPweak}.
We introduce auxiliary quantities
\begin{align}
\label{eq:ra}
\rTa &= c \HF u_{h,n} - \lambda_{h,n} \beta_1 \HF  u_{h,n} +  (\cA \nabla \HF) \cdot \nabla u_{h,n},
\\
\label{eq:ga}
\gTa &= \alpha \HF u_{h,n} - \lambda_{h,n} \beta_2 \HF u_{h,n} .
%    \rTa &= \left\{ \begin{array}{lll}
%     \displaystyle  \HF \lambda_{h,n} u_{h,n} - c \HF u_{h,n} - \cA \nabla \HF \nabla u_{h,n}, \quad \text{in Friedrichs' case}, \\
%     \displaystyle  \HF \lambda_{h,n} u_{h,n} - \cA \nabla \HF \nabla u_{h,n}, \quad \text{in Poincar\'e case}, \\
%     \displaystyle  - c \HF u_{h,n} - \cA \nabla \HF \nabla u_{h,n}, \quad \text{in trace case}.
%     \end{array} \right.\\
%    \gN &= \left\{ \begin{array}{lll}
%     \displaystyle  - \alpha u_{h,n}, \quad \text{in Friedrichs' case}, \\
%     \displaystyle  0, \quad \text{in Poincar\'e case}, \\
%     \displaystyle  \lambda_{h,n} u_{h,n} - \alpha u_{h,n}, \quad \text{in trace case}.
%     \end{array} \right.
\end{align}
Note that these quantities are defined in such a way that
\begin{equation}
  \label{eq:ragN}
  a(u_{h,n}, \HF) - \lambda_{h,n} b(u_{h,n},\HF) = \int_{\Oma} \rTa \dx + \sum_{\gamma\in\cEaBN} \int_\gamma \gTa \dx[s]
  \quad\forall \ba\in\NT.
\end{equation}

The local flux reconstruction is defined in the Raviart--Thomas spaces on patches $\Oma$ with suitable
boundary conditions. We introduce the space
\begin{multline}
\label{eq:Wa0}
% \bWT(\Oma) = \left\{  \bwT \in \Hdiv[\Oma]:
%     \bwT|_K \in [P_p(K)]^2 \oplus \boldsymbol{x}P_p(K) \quad \forall K \in \Ta
%   \right \}
\bWa^0 = \left\{
    \bwT \in \Hdiv[\Oma] : \bwT|_K \in \RT(K) \quad\forall K \in \Ta
    \right. \\ \left.
    \quad\text{and}\quad
    \bwT \cdot \bn_{\gamma} =
      0 \text{ on edges } \gamma \in \cEaBE \cup \cEaBN
  \right\}
\end{multline}
and the affine set
\begin{multline}
  \label{eq:Wa}
  \bWa = \left\{
    %\bwT \in \bWT(\Oma) :
    \bwT \in \Hdiv[\Oma] : \bwT|_K \in \RT(K) \quad\forall K \in \Ta, \quad
    \bwT \cdot \bn_{\gamma} = 0
    \right. \\ \left.
    \text{ on edges } \gamma \in \cEaBE
    \text{ and }
    \bwT \cdot \bn_{\gamma} = -\Pi_{\gamma}(\gTa) \text{ on edges } \gamma \in \cEaBN
%       \left\{ \begin{array}{ll}
%       0 & \text{ on edges } \gamma \in \cEaBE
%       \\
%       \Pi_{\gamma}(\gTa) & \text{ on edges } \gamma \in \cEaBN
%       \end{array}
%     \right.
  \right\}.
\end{multline}
The symbol $\Pi_{\gamma}$ stands for the $L^2(\gamma)$-orthogonal projection onto the space $P_p(\gamma)$ of polynomials of degree at most $p$ on the edge $\gamma$.
We also define the space $P_p(\Ta) = \{ \vT \in L^2(\Oma) : \vT|_K \in P_p(K) \ \forall K\in\Ta \}$ of piecewise polynomial and in general discontinuous functions.
Further, we introduce the space
\begin{equation} \label{eq:Ppast}
  \PpTast = \left\{ \begin{array}{ll}
    \{ \vT \in P_p(\Ta): \int_{\Oma} \vT \dx = 0 \}, \quad \text{for} \ \ba \in \NTI \cup \NTN, \\
    P_p(\Ta) , \quad \text{for} \ \ba \in \NTD.%\\
    \end{array} \right.
\end{equation}

Using these spaces, we define the flux reconstruction $\tqT \in \bWT$ as the sum
\begin{equation} \label{eq:defq}
  \tqT = \sum_{\ba \in \NT} \tqa,
\end{equation}
where $\tqa \in \bWa$ together with $\da \in \PpTast$ solves the mixed finite element problem
\begin{alignat}{2}
%\label{min_prob_qa}
       (\cA^{-1} \tqa, \bwT)_{\Oma} - (\da, \ddiv \bwT)_{\Oma} &= (\HF \nabla u_{h,n}, \bwT)_{\Oma}
         &\quad& \forall \bwT \in \bWa^0, \label{min_prob_qa_1}\\
       (\ddiv \tqa, \vT)_{\Oma} &= (\rTa, \vT)_{\Oma} &\quad &\forall \vT \in \PpTast. \label{min_prob_qa_2}
\end{alignat}
Let us note that this mixed finite element problem is equivalent to the minimization of
$\left\|\HF \cA^{\frac{1}{2}} \nabla u_{h,n} - \cA^{-\frac{1}{2}} \tsa \right\|_{\Oma}$
over all $\tsa \in \bWa$ satisfying the constraint $\ddiv \tsa = \Pi_{p}(\rTa)$ in $\Oma$,
where $\Pi_{p}$ denotes the $L^2(\Oma)$-orthogonal projection onto $P_p(\Ta)$.

% \cred{Note that the existence and uniqueness of problem \eqref{min_prob_qa_1}--\eqref{min_prob_qa_2} follows by the Neumann compatibility condition~\eqref{eq:Neumequilib} and the classical results in mixed finite element method  (proof is recalled in, e.g., \cite{Voh_un_apr_apost_MFE_10}).}
The flux reconstruction \eqref{eq:defq} is defined in such a way that quantities $r$ and $g$ vanish. We prove this fact below in Lemma~\ref{le:etaR=0}. However, first, we need to prove that identity \eqref{min_prob_qa_2} can be actually tested by any polynomial.
\begin{lemma}
Let $\tqa \in \bWa$ and $\da \in \PpTast$ be a solution of problem \eqref{min_prob_qa_1}--\eqref{min_prob_qa_2}. Then
\begin{equation}
\label{eq:testPp}
 (\ddiv \tqa, \vT)_{\Oma} = (\rTa, \vT)_{\Oma} \quad \forall \vT \in \PpT.
\end{equation}
\end{lemma}
\begin{proof}
Notice that if $\ba \in \NTD$ then there is nothing to prove due to definition~\eqref{eq:Ppast}.
If $\ba \in \NTI \cup \NTN$ then we can use $\HF$ as a test function in \eqref{eq:EPFEM}.
Consequently, identity \eqref{eq:ragN} and definition \eqref{eq:Wa} imply
\begin{equation}
\label{eq:Neumequilib}
  (\rTa,1)_{\Oma}
  = -\sum_{\gamma\in\cEaBN}(\gTa, 1)_\gamma%{\GammaN\cap\partial\Oma}
  = (\tqa \cdot \bn_{\partial\Oma}, 1)_{\partial\Oma}
  = (\ddiv \tqa, 1)_{\Oma}.
\end{equation}
\end{proof}

Let us note that for vertices $\ba \in \NTI \cup \NTN$ the problem~\eqref{min_prob_qa_1}--\eqref{min_prob_qa_2}
corresponds to a pure Neumann problem for $\da$.
This problem is solvable, because the corresponding equilibrium condition is exactly \eqref{eq:Neumequilib}.
In addition, its solution is unique thanks to the fact that
the space $\PpTast$ does not contain constant functions. For $\ba \in \NTD$, the problem~\eqref{min_prob_qa_1}--\eqref{min_prob_qa_2} corresponds to a well posed Dirichlet--Neumann problem
and the space $\PpTast$ contains constant functions.
%Alternatively, we could consider $\da \in \PpT$ and test \eqref{eq:min_prob_qa_2} by any
%$\vT \in \PpT$.
%More details about the existence and uniqueness of problem~\eqref{min_prob_qa_1}--\eqref{min_prob_qa_2} can be found for example in \cite{Voh_un_apr_apost_MFE_10}.
The existence and uniqueness of problem~\eqref{min_prob_qa_1}--\eqref{min_prob_qa_2} can be rigorously proved in the same way as in \cite{Voh_un_apr_apost_MFE_10}.

% Concerning the solvability of system~\eqref{min_prob_qa_1}--\eqref{min_prob_qa_2}, it is important
% to appreciate the following fact.
% \begin{lemma} \label{NCC}
%     %Let $(\lambda_{h,n}, u_{h,n}) \in \VT$ be the approximation of the smallest (positive in Poincar\'e case) eigenvalue and corresponding eigenfunction given by~\eqref{eq:discreigp}.
%     %The Neumann compatibility condition for the pure Neumann problem
%     Equation~\eqref{min_prob_qa_2} can be tested by all $\vT\in\PpT$, because
%     %
%     \begin{equation} \label{NCC_homo}
%        (\rTa, 1)_{\Oma} = 0 \, (= - ( \tqa {\cdot} \bn, 1 )_{\pt \Ta}) \text{ for } \ba \in \NTI,
%     \end{equation}
%     %
%     \begin{equation} \label{NCC_inhomo}
%        - (\rTa, 1)_{\Oma} = \left ( \Pi_p(\HF \gN), 1 \right )_{\pt \Ta \cap \cETN} ( = ( \tqa {\cdot} \bn, 1 )_{\pt \Ta}) \text{ for } \ba \in \NTN.
%     \end{equation}
% \end{lemma}
% \cred{WHY $- ( \tqa {\cdot} \bn, 1 )_{\pt \Ta}$ IS THERE? WHY IT IS IMPORTANT?}
% \begin{proof}
%     Both equalities~\eqref{NCC_homo} and \eqref{NCC_inhomo} follow
%     from \eqref{eq:ragN}.
%     %by taking $\vT := \HF \in \VT$ and $(\lamT, \uT) := (\lambda_{h,n}, u_{h,n})$ in~\eqref{eq:discreigp}.
% \end{proof}

Finally, we present the result that quantities $r$ and $g$, see \eqref{eq:Frg}, vanish for the described flux reconstruction $\tqT$.
\begin{lemma}
\label{le:etaR=0}
%Let $\tqa \in \RTNpN(\Ta)$ be the solution to problem \eqref{min_prob_qa_1}--\eqref{min_prob_qa_2}.
Let $\tqT \in \bWT$ be given by \eqref{eq:defq} and the problem \eqref{min_prob_qa_1}--\eqref{min_prob_qa_2}.
Then
\begin{align}
 \label{eq:etaR=0}
 c u_{h,n} - \lambda_{h,n} \beta_1 u_{h,n} - \ddiv \tqT &= 0
 \quad\text{a.e. in }\Omega,
 \\
 \label{eq:etaN=0}
 \alpha u_{h,n} - \lambda_{h,n} \beta_2 u_{h,n} + \tqT \cdot \bn_\Omega &= 0
 \quad\text{a.e. on }\GammaN.
\end{align}
% \begin{align}
%  \label{eq:etaR=0}
%  \resi &= \norm{\lambda_{h,n} \beta_1 u_{h,n} - c u_{h,n} + \ddiv \tqT}_K = 0,
%  \\
%  \label{eq:etaN=0}
%  \neubi &= \sum_{\gamma \in \cETK \cap \cETN} \norm{ \alpha u_{h,n} - \lambda_{h,n} \beta_2 u_{h,n} + \tqT|_K \cdot \bn_K }_{L^2(\gamma)} = 0.
% \end{align}
\end{lemma}
\begin{proof}
   Let us set $r_h = c u_{h,n} - \lambda_{h,n} \beta_1 u_{h,n} - \ddiv \tqT$.
   Clearly, $r_h|_K \in P_p(K)$ for all $K\in\cT_h$.
   Using the decomposition of unity $\sum_{\ba \in \NTK} \HF = 1$, notation \eqref{eq:ra},
   and equality \eqref{eq:testPp}, we obtain
   \begin{align*} %\label{eval_res_indic}
      \norm{r_h}^2
      &= \sum_{\ba \in \NTK} \left( c \HF u_{h,n} - \lambda_{h,n} \beta_1 \HF u_{h,n} - \ddiv \tqa,
               r_h \right)_{\Oma}
      \\
      &= \sum_{\ba \in \NTK} \left( \rTa - (\cA \nabla\HF)\cdot\nabla u_{h,n} - \ddiv \tqa,
               r_h \right)_{\Oma}
      \\               
      &= - \sum_{\ba \in \NTK} \left( (\cA \nabla\HF)\cdot \nabla u_{h,n}, r_h \right)_{\Oma} = 0.
   \end{align*}
   Thus, $r_h$ vanishes almost everywhere in $\Omega$.

  To prove the second statement, we set $g_h =  \alpha u_{h,n} - \lambda_{h,n} \beta_2 u_{h,n} + \tqT \cdot \bn_\Omega$. Let $\gamma$ be an arbitrary edge on the Neumann boundary. Clearly, $g_h|_\gamma \in P_p(\gamma)$.
  The decomposition of unity $\sum_{\ba \in \NTGa} \HF = 1$ with $\NTGa$ being the set of the two end-points of the edge $\gamma$ and definitions~\eqref{eq:Wa} and \eqref{eq:ga} then give
  \begin{multline}\label{eval_trace_Neum_indic}
    \norm{ g_h }_\gamma^2
    = \sum_{\ba \in \NTGa} (\alpha \HF u_{h,n} - \lambda_{h,n} \beta_2 \HF u_{h,n} + \tqa {\cdot} \bn_\gamma, g_h)_\gamma
    \\
    = \sum_{\ba \in \NTGa} (\alpha \HF u_{h,n} - \lambda_{h,n} \beta_2 \HF u_{h,n} - \Pi_\gamma(\gTa), g_h)_\gamma
    = 0.
  \end{multline}
  Thus, $g_h$ vanishes almost everywhere on $\gamma$ and, hence, on $\GammaN$.
\end{proof}

The following corollary summarizes the fact that properties \eqref{eq:etaR=0} and \eqref{eq:etaN=0} of the flux reconstruction $\tqT$ simplify
the error estimator presented in Theorem~\ref{th:estimator} considerably.

\begin{corollary}
\label{co:normw}
Let the flux reconstruction $\tqT \in \bWT$ be given by \eqref{eq:defq} and problem \eqref{min_prob_qa_1}--\eqref{min_prob_qa_2}.
Let $\lambda_{h,n} \in \R$ and $u_{h,n} \in \VT\setminus\{0\}$ satisfy \eqref{eq:EPFEM}.
Further, let $w_n \in V$ be given by \eqref{eq:defw} with bilinear forms defined in \eqref{eq:blf}--\eqref{eq:blfb}
and with $\lambda_{*,n} = \lambda_{h,n}$ and $u_{*,n} = u_{h,n}$.
Then
\begin{equation}
  \label{eq:etasimp}
  \norm{w_n}_a \leq \eta_n,
  \quad\text{where}\quad
  \eta_n^2 = \sum_{K\in\cT_h} F_K^2 = \norm{\cA^{1/2} \left( \nabla u_{h,n} - \cA^{-1} \tqT \right) }_{L^2(\Omega)}^2.
\end{equation}
% where $\eta_n = \norm{\cA^{1/2} \left( \nabla u_{h,n} - \cA^{-1} \tqT \right) }_{L^2(\Omega)}$ is given by \eqref{eq:etasimp}.
% %$\eta = \norm{\nabla u_{h,n} - \cA^{-1} \tqT}_{\cA}$.
\end{corollary}
\begin{proof}
The statement follows immediately from Theorem~\ref{th:estimator} and properties \eqref{eq:etaR=0} and \eqref{eq:etaN=0}.
\end{proof}

\begin{remark}
% We note that $\tqT$ will be defined in such a way that both $r$ and $g$ vanish, see their definitions in \eqref{eq:Frg}. Therefore, the complex formula for the estimator $\eta_n$ in Theorem~\ref{th:estimator} simplifies to
% \begin{equation}
%   \label{eq:etasimp}
%   \eta_n^2 = \sum_{K\in\cT_h} F_K^2 = \norm{\cA^{1/2} \left( \nabla u_{h,n} - \cA^{-1} \tqT \right) }_{L^2(\Omega)}^2.
% \end{equation}
In floating-point arithmetics, we cannot solve problem \eqref{min_prob_qa_1}--\eqref{min_prob_qa_2} exactly due to round-off errors.
Consequently, hypotheses of Corollary~\ref{co:normw} are not satisfied, and the bound \eqref{eq:etasimp} is not guaranteed.
The point is that quantities $r$ and $g$ do not vanish exactly in this case.
In order to overcome this issue, we recommend to use Theorem~\ref{th:estimator} with $\bq = \tqT$.
For example, if $r\neq 0$, $c = 0$, and $\beta_1 > 0$ in $\Omega$ then $r_{1,K}=\infty$, $r_{2,K}=\infty$, and $r_{3,K} < \infty$ for all $K\in\cT_h$. Consequently, $M_K=\infty$, $\cT_h^{++} = \cT_h$, and $\cT_h^{+0}=\cT_h^{0} = \emptyset$. If we consider for simplicity $\GammaN = \emptyset$ then $Y_K = F_K$, % = \norm{\cA^{1/2} \left( \nabla u_{h,n} - \cA^{-1} \tqT \right) }_K$,
$R_K = r_{3,K}$, % = \norm{\beta_{1,K}^{-1/2} \left( c_K u_{h,n} - \lambda_{h,n} \beta_{1,K} u_{h,n}  - \ddiv \tqT \right) }_K $
and estimator \eqref{eq:eta} reduces to
\begin{equation}
\label{eq:etaguar}
%   \eta_n = \norm{\cA^{1/2} \left( \nabla u_{h,n} - \cA^{-1} \tqT \right) }_{L^2(\Omega)}
%   + \underline{\lambda}_1^{-1/2} \norm{\beta_1^{-1/2} \left( c u_{h,n} - \lambda_{h,n} \beta_1 u_{h,n}  - \ddiv \bqT \right) }_{L^2(\Omega)}.
  \eta_n^2 = \sum_{K\in\cT_h} F_K^2
  + \underline{\lambda}_1^{-1/2} \sum_{K\in\cT_h} r_{3,K}^2.
\end{equation}
This estimator provides a guaranteed upper bound on $\norm{w_n}_a$ even if $r$ does not equal to zero exactly.
However, it has to be said that in practical computations the quantity $r$ and, thus, all quantities $r_{3,K}$ are typically on the level of machine precision. Therefore, %the dominant term is is $\norm{\cA^{1/2} F }_{L^2(\Omega)}$ and
the difference of estimators given by \eqref{eq:etasimp} and \eqref{eq:etaguar} is often negligible.
% Hence, Theorem~\ref{th:estimator} enables to compute the guaranteed upper bound $\eta_n$ on $\norm{w_n}_a$ even if $r \neq 0$ and $g \neq 0$. \cred{Tedy $\eta_n$ nebude dano jen formuli vyse, ne??}
\end{remark}

\section{Numerical example in dumbbell shape domain}
\label{se:numex}

% .... The guaranteed upper bound provided by Lemma~\ref{le:normw} immediately enables to use Theorems~\ref{th:Weinlowerbound} and \ref{th:Katobound} to compute lower bounds \eqref{eq:Wein} and \eqref{eq:Kato} on the corresponding eigenvalue.

This section illustrates the numerical performance of lower bounds $\ell_n$ and $L_n$ given by
\eqref{eq:Wein} and \eqref{eq:Kato}, respectively.
These bounds depend on quantities $\eta_i$, which are guaranteed bounds on representatives of residuals, see \eqref{eq:abscompl}.
We compute these quantities by the local flux reconstruction procedure described in Section~\ref{se:flux}.

We will compute two-sided bounds on the first ten eigenvalues of the Laplacian in a dumbbell shaped domain \cite{TreBet2006} with homogeneous Dirichlet boundary conditions.
Thus, we consider problem \eqref{eq:EPstrong} in the domain $\Omega$ showed in Figure~\ref{fi:dumbbell} (left)
with $\GammaN = \emptyset$, $\cA$ being the identity matrix, and with constant coefficients $c=0$, $\alpha = 0$, $\beta_1 = 1$, $\beta_2 = 0$.
We solve this problem by the standard conforming finite element method with piecewise linear and continuous trial and test functions, i.e., we consider the finite element space \eqref{eq:Vh} with $p=1$ and the finite element approximation given by \eqref{eq:EPFEM}.
The computed eigenvalues $\lambda_{h,n}$ are upper bounds on the corresponding exact eigenvalues.
Lower bounds are computed in two ways. First, we use lower bound $L_n$ only and
the quantity $\nu$ satisfying \eqref{eq:nucond} is obtained by the homotopy method.
Second, we combine both lower bounds $\ell_n$ and $L_n$.

%mark=at position 0 with {\arrowreversed[scale=2,thin,black]{>}}
\tikzset{
  big arrowhead/.style={
    decoration={markings, mark=at position 1 with {\arrow[scale=1.5,thin,black]{>}} },
    postaction={decorate},
    shorten >=0.4pt}}
\begin{figure}
\begin{center}
\begin{tikzpicture}[scale=0.6]
\draw [thick] (0,0)--(4,0)--(4,1.5)--(5,1.5)--(5,0)--(9,0);
\draw [thick] (9,0)--(9,4)--(5,4)--(5,2.5)--(4,2.5)--(4,4)--(0,4)--(0,0);
%\node [below] at (5.5,0) {$\GammaD$};
%\node [above] at (5.5,3) {$\GammaN$};
%\draw [<-{>[scale=2]}, thin] (0,0.5)--(3,0.5);
\draw [big arrowhead, thin] (0,3.5)--(4,3.5);
\draw [big arrowhead, thin] (4,3.5)--(0,3.5);
\node [below] at (2.5,3.5) {$\pi$};
\draw [big arrowhead, thin] (4,3.5)--(5,3.5);
\draw [big arrowhead, thin] (5,3.5)--(4,3.5);
\node [above] at (4.5,3.5) {$\displaystyle\frac{\pi}{4}$};
\draw [big arrowhead, thin] (5,3.5)--(9,3.5);
\draw [big arrowhead, thin] (9,3.5)--(5,3.5);
\node [below] at (7,3.5) {$\pi$};
\draw [big arrowhead, thin] (0.5,0)--(0.5,4);
\draw [big arrowhead, thin] (0.5,4)--(0.5,0);
\node [right] at (0.5,2) {$\pi$};
\draw [big arrowhead, thin] (4.8,1.5)--(4.8,2.5);
\draw [big arrowhead, thin] (4.8,2.5)--(4.8,1.5);
\node [right] at (4.8,2) {$\displaystyle\frac{\pi}{4}$};
%\node [right] at (1,1) {$\Omega = \Omega^{(4)}$};
\end{tikzpicture}
\qquad
\raisebox{-0.2mm}{\includegraphics[width=0.460\textwidth]{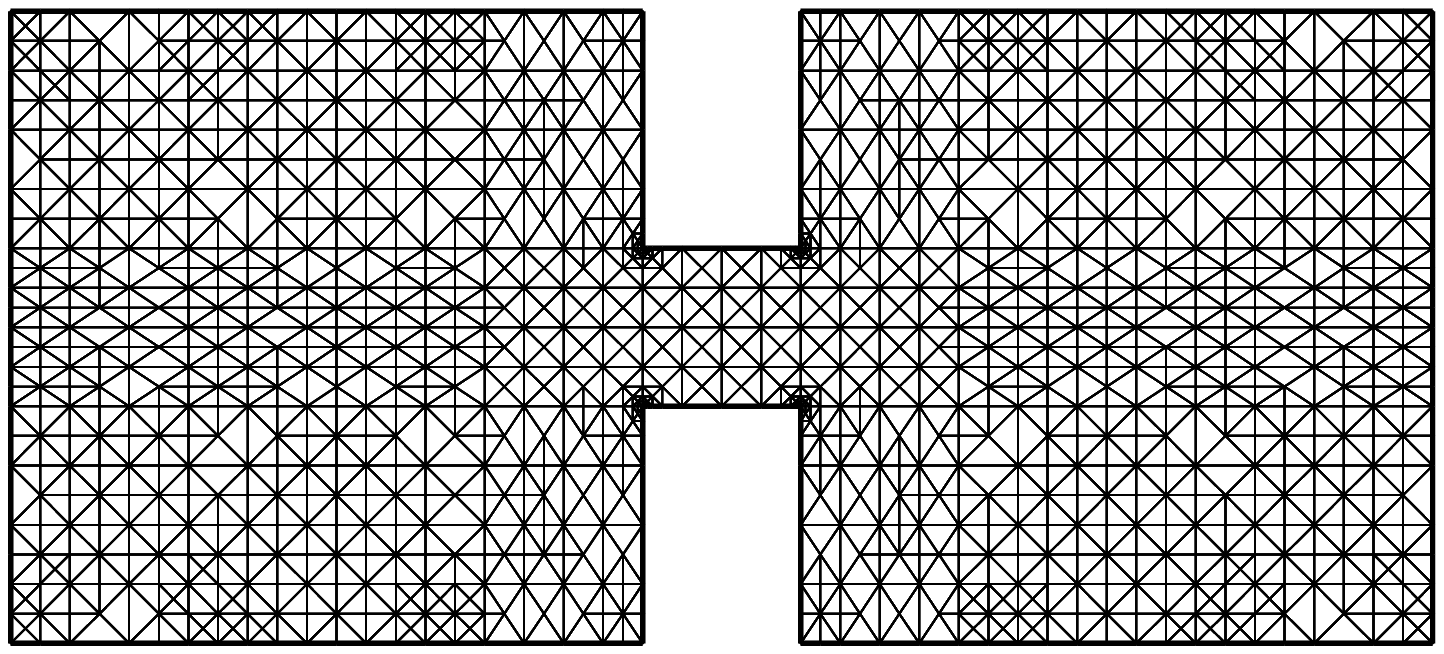}}
\end{center}
\caption{\label{fi:dumbbell}
The left panel shows the dimensions of the dumbbell shaped domain $\Omega=\Omega^{(4)}$.
The right panel presents the adaptively refined mesh after 20 adaptive steps.
}
\end{figure}

Due to singularities of eigenfunctions in reentrant corners, we utilize a mesh adaptive algorithm. This algorithm follows the standard loop:
\begin{itemize}
\item \texttt{SOLVE}.
Given a mesh $\cT_h$, we compute finite element approximate eigenpairs $\lambda_{h,n}$, $u_{h,n}$, for $n=r,\dots,s$, see \eqref{eq:EPFEM}.
%We compute the first $s$ eigenpairs on the given a mesh $\cT_h$ by the finite element method \eqref{eq:EPFEM}.
\item \texttt{ESTIMATE}. We solve local problems \eqref{min_prob_qa_1}--\eqref{min_prob_qa_2}, reconstruct the flux \eqref{eq:defq}, and compute error estimator \eqref{eq:eta} for $n=r,\dots,s$.
%\item \texttt{LOWER BOUNDS}.
We check if the left-hand side inequality of condition \eqref{eq:nucond} is satisfied.
%the computed approximate eigenvalues are sufficiently accurate such that
If so, we compute the lower bound \eqref{eq:Kato} for $n=r,\dots,s$. In order to obtain as accurate estimates as possible, we evaluate these lower bounds recursively, as described at the end of Section~\ref{se:lbounds}.
%,  meshes rather than on specific meshes for individual eigenvalues. This saves certain overhead computations connected with the construction of patches of elements.
\item \texttt{MARK}.
We use quantities $F_K$, see \eqref{eq:defFK} and \eqref{eq:etasimp}, to mark elements for refinement. For each element we have $s-r+1$ quantities $F_K$ corresponding to eigenvalues $\lambda_n$, $n=r,\dots,s$. We use their maximum as the error indicator for the mesh refinement.
We mark elements by the bulk criterion \cite{Dorfler:1996}.
\item \texttt{REFINE}.
We refine the marked elements by the newest vertex bisection algorithm, see e.g. \cite{SchSie2005}.
\end{itemize}
In order to compare the accuracy of computed eigenvalue bounds, we stop this adaptive algorithm as soon as the number of degrees of freedom exceeds 750\,000.

\begin{figure}
\begin{tikzpicture}[scale=0.25]
\draw [thick] (0,0)--(9,0);
\draw [thick] (9,0)--(9,4)--(0,4)--(0,0);
\node [right] at (1,2) {$\Omega^{(0)}$};
\end{tikzpicture}
\
\begin{tikzpicture}[scale=0.25]
\draw [thick] (0,0)--(4,0)--(4,0.375)--(5,0.375)--(5,0)--(9,0);
\draw [thick] (9,0)--(9,4)--(5,4)--(5,3.625)--(4,3.625)--(4,4)--(0,4)--(0,0);
\draw [big arrowhead, thin] (4.8,0.375)--(4.8,3.625);
\draw [big arrowhead, thin] (4.8,3.625)--(4.8,0.375);
\node [right] at (4.8,2) {$\frac{13\pi}{16}$};
\node [right] at (0.9,2) {$\Omega^{(1)}$};
\end{tikzpicture}
\
\begin{tikzpicture}[scale=0.25]
\draw [thick] (0,0)--(4,0)--(4,0.75)--(5,0.75)--(5,0)--(9,0);
\draw [thick] (9,0)--(9,4)--(5,4)--(5,3.25)--(4,3.25)--(4,4)--(0,4)--(0,0);
\draw [big arrowhead, thin] (4.8,0.75)--(4.8,3.25);
\draw [big arrowhead, thin] (4.8,3.25)--(4.8,0.75);
\node [right] at (4.8,2) {$\frac{5\pi}{8}$};
\node [right] at (0.8,2) {$\Omega^{(2)}$};
\end{tikzpicture}
\
\begin{tikzpicture}[scale=0.25]
\draw [thick] (0,0)--(4,0)--(4,1.125)--(5,1.125)--(5,0)--(9,0);
\draw [thick] (9,0)--(9,4)--(5,4)--(5,2.875)--(4,2.875)--(4,4)--(0,4)--(0,0);
\draw [big arrowhead, thin] (4.8,1.125)--(4.8,2.875);
\draw [big arrowhead, thin] (4.8,2.875)--(4.8,1.125);
\node [right] at (4.8,2) {$\frac{7\pi}{16}$};
\node [right] at (0.65,2) {$\Omega^{(3)}$};
\end{tikzpicture}
\
\begin{tikzpicture}[scale=0.25]
\draw [thick] (0,0)--(4,0)--(4,1.5)--(5,1.5)--(5,0)--(9,0);
\draw [thick] (9,0)--(9,4)--(5,4)--(5,2.5)--(4,2.5)--(4,4)--(0,4)--(0,0);
\draw [big arrowhead, thin] (4.8,1.5)--(4.8,2.5);
\draw [big arrowhead, thin] (4.8,2.5)--(4.8,1.5);
\node [right] at (4.8,2) {$\frac{\pi}{4}$};
\node [right] at (0.75,2) {$\Omega^{(4)}$};
\end{tikzpicture}
\caption{\label{fi:homotopy}
The five domains we use for the homotopy transition between the rectangle $\Omega^{(0)}$ and the dumbbell shaped domain $\Omega^{(4)}$.
}
\end{figure}
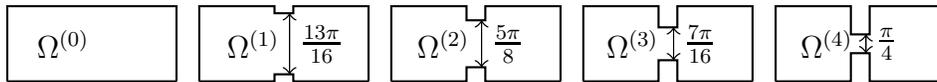

\subsection*{Bound $L_n$ with the homotopy method.}
The quantity $\nu$ satisfying \eqref{eq:nucond} is computed by the homotopy method described in detail in \cite{Plum1990}. In particular, we use the specific procedure from \cite[Sec.~6.1, par.~b)]{Plum1991}.
We consider a sequence of five domains $\Omega^{(\ihom)}$, $\ihom=0,1,\dots,4$. These domains are interiors of the union of squares $(0,\pi)^2$, $(5\pi/4,9\pi/4)\times(0,\pi)$, and the rectangle $[\pi,5\pi/4]\times(y_1,\pi-y_1)$, where $y_1 = 3\ihom\pi/32$ and $\ihom = 0,1,\dots,4$, see Figure~\ref{fi:homotopy}. Notice that $\Omega^{(0)} = (0,9\pi/4)\times(0,\pi)$ is a rectangle and $\Omega^{(4)} = \Omega$ is the targeted dumbbell shaped domain.
Moreover, these domains are nested $\Omega^{(4)} \subset \Omega^{(3)} \subset \cdots \subset \Omega^{(0)}$ and thus the eigenvalues on the larger domain are below the corresponding eigenvalues on the smaller domain. More accurately, if $\lambda^{(\ihom)}_n$ denotes the $n$-th eigenvalue of the Laplace eigenvalue problem in the domain $\Omega^{(\ihom)}$ then
the Courant minimax principle \cite{BabOsb:1991} implies that
\begin{equation}
  \label{eq:hombound}
  \lambda^{(\ihom-1)}_n \leq \lambda^{(\ihom)}_n,
  \quad\text{for } \ihom = 1,2,3,4 \text{ and } n=1,2,\dots.
\end{equation}

Eigenvalues on the rectangle $\Omega^{(0)}$ are known analytically:
$$
  \lambda^{(0)}_{i,j} = 16 i^2 / 81 + j^2, \quad i,j=1,2,\dots.
$$
One of these values can be chosen for $\nu$ to compute lower bounds on the domain $\Omega^{(1)}$.
Due to the large spectral gap $\lambda^{(0)}_{21} - \lambda^{(0)}_{20}$, see Table~\ref{ta:res}, we choose
$\nu = 16.1111 < 145/9 = \lambda^{(0)}_{6,3} = \lambda^{(0)}_{21}$.
%the twenty first eigenvalue $\lambda^{(0)}_{21} = \lambda^{(0)}_{6,3} = 145/9 > 16.1111$. Thus, $\nu = 16.1111$
%This $\nu$ is guaranteed to be below the twenty first eigenvalue on the domain $\Omega^{(1)}$, i.e. by \eqref{eq:hombound} we have $\nu < \lambda^{(1)}_{21}$.
By \eqref{eq:hombound} we have $\nu < \lambda^{(1)}_{21}$.
Hence, the right-hand side inequality of condition \eqref{eq:nucond} is satisfied for $s=20$
and we can use \eqref{eq:Kato}, see Theorem~\ref{th:Katobound}, to compute lower bounds
on $\lambda^{(1)}_1, \dots, \lambda^{(1)}_{20}$.

We compute these lower bounds by the adaptive algorithm described above using initial meshes depicted in Figure~\ref{fi:mesh1}. Note that in order to evaluate $\eta_n^{(\mathrm{b})}$ in \eqref{eq:eta} we use $\underline{\lambda}_1 = \lambda^{(0)}_1 = 97/81$ computed analytically for the rectangle $\Omega^{(0)}$. This value is guaranteed to be below $\lambda^{(\ihom)}_1$ for all $\ihom = 1,2,3,4$, see \eqref{eq:hombound}.

\begin{figure}%
\includegraphics[width=0.23\textwidth]{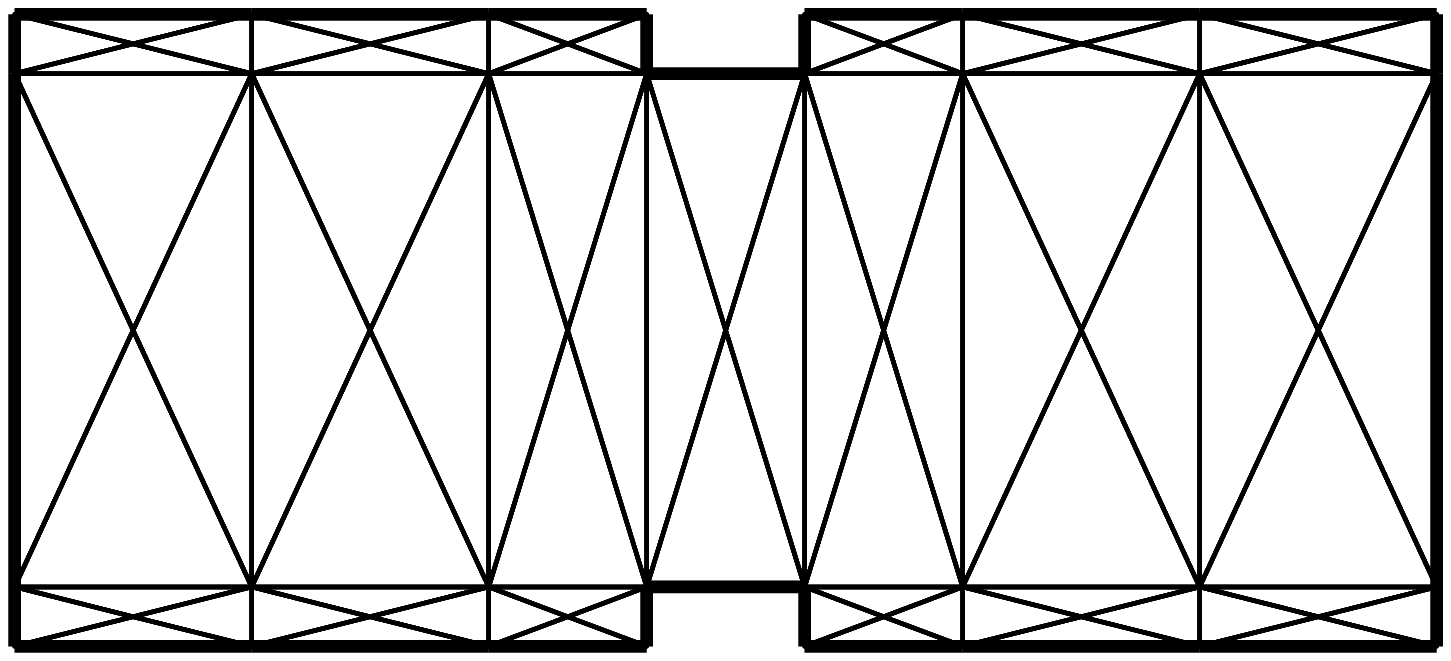}\ \
\includegraphics[width=0.23\textwidth]{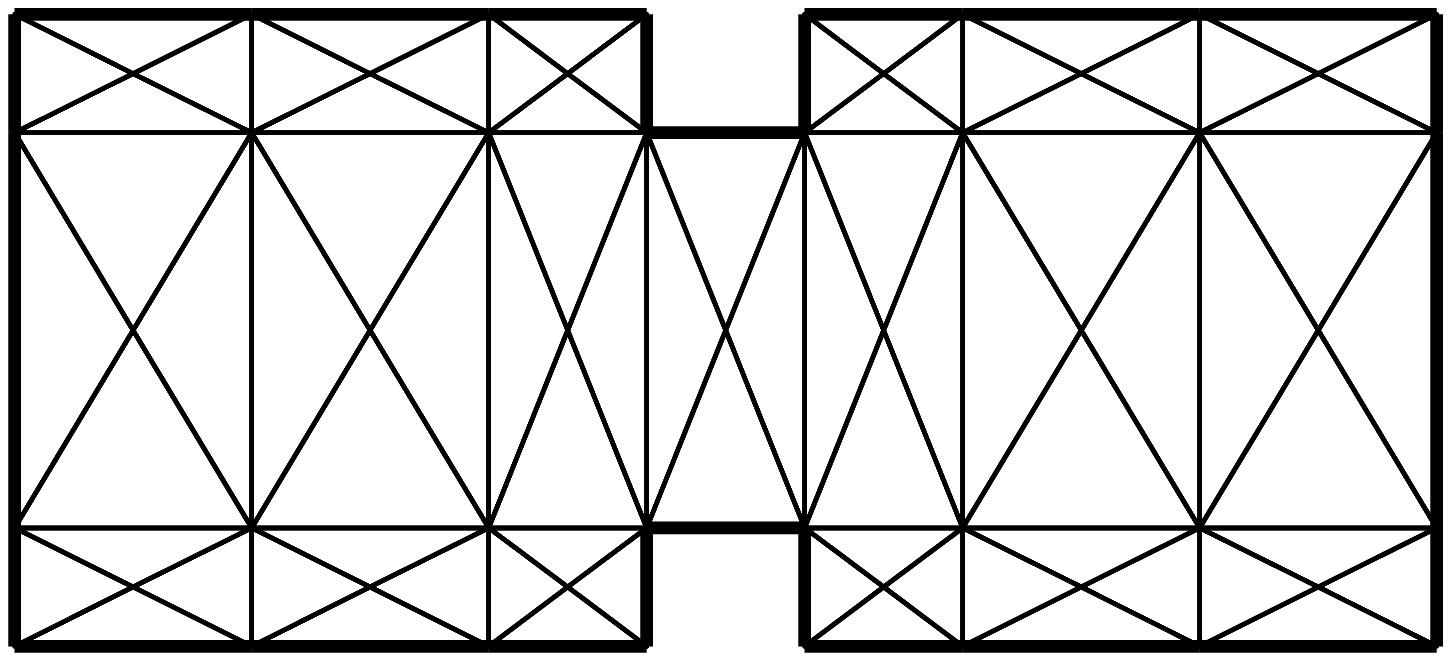}\ \
\includegraphics[width=0.23\textwidth]{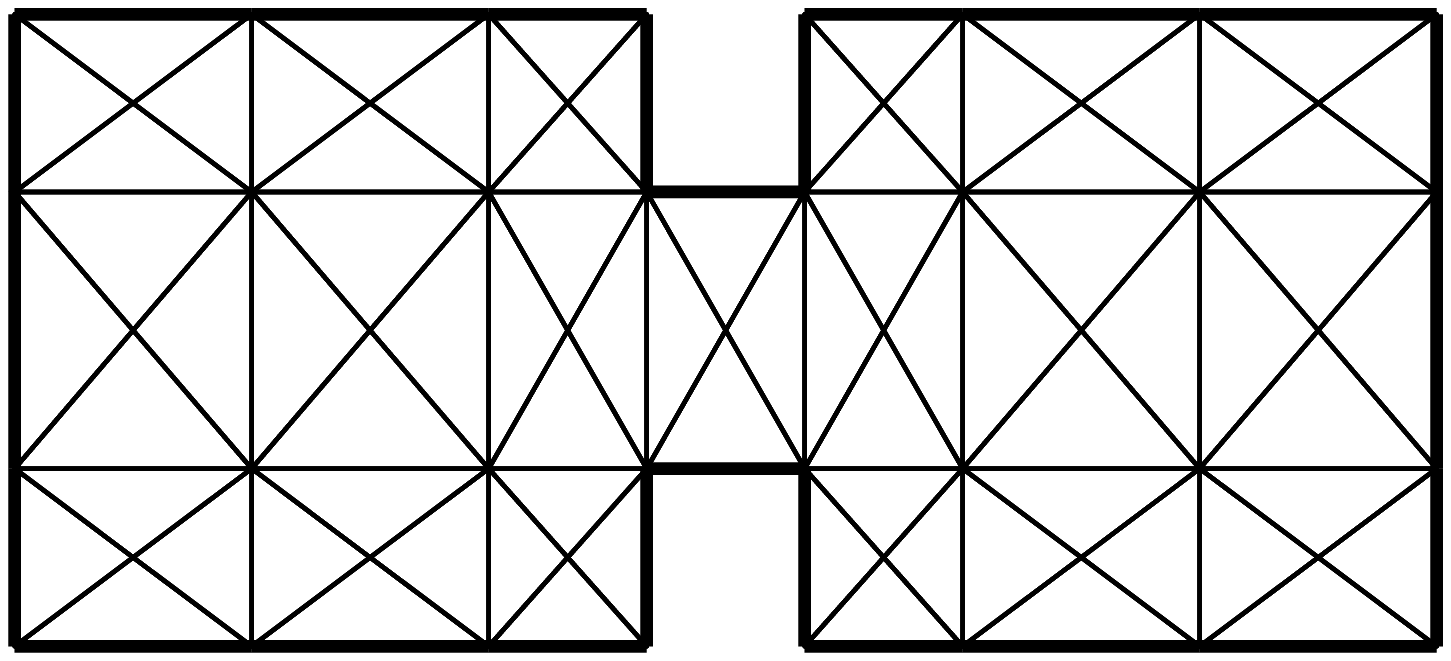}\ \
\includegraphics[width=0.23\textwidth]{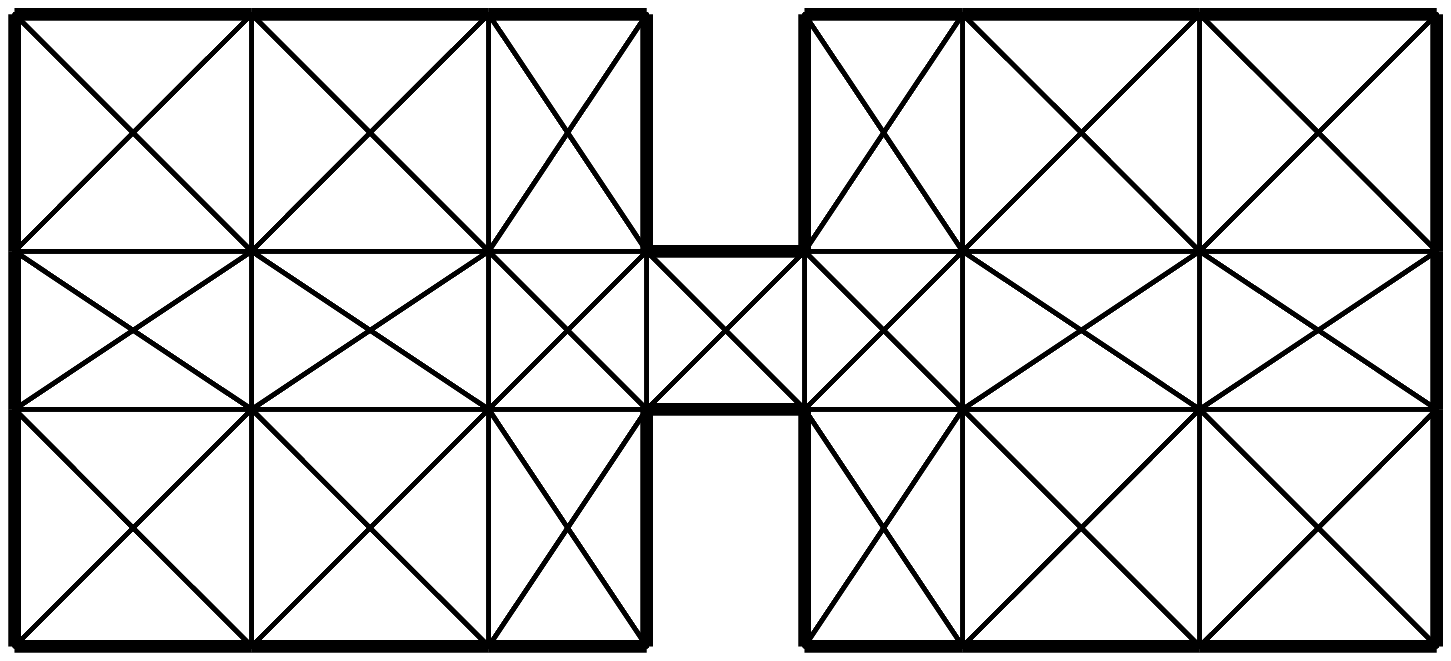}\ \
\caption{\label{fi:mesh1}
Initial meshes on domains $\Omega^{(1)}$, \dots, $\Omega^{(4)}$, respectively.
}
\end{figure}

%The resulting two-sided bounds on the first twenty eigenvalues are presented in column $\Omega^{(1)}$ in Table~\ref{ta:res}.
To continue the homotopy method, we set $\nu = L^{(1)}_{20}$, where $L^{(1)}_{20}$ is the lower bound on $\lambda^{(1)}_{20}$ for the domain $\Omega^{(1)}$. Using this value, we
% As soon as this happens, we have accurate lower bounds $\tilde L^{(1)}_1$, \dots, $\tilde L^{(1)}_{20}$ on the first twenty eigenvalues of the \tilde Laplacian in the domain $\Omega^{(1)}$. Now, we set $\nu = \tilde L^{(1)}_{20}$,
repeat the adaptive process for $\Omega^{(2)}$ and compute lower bounds $L^{(2)}_1$, \dots, $L^{(2)}_{19}$ on the first nineteen eigenvalues for the domain $\Omega^{(2)}$. Then we choose $\nu = L^{(2)}_{18}$ and compute by the same procedure lower bounds $L^{(3)}_1$, \dots, $L^{(3)}_{17}$ for the domain $\Omega^{(3)}$. Note that we did not choose $L^{(2)}_{19}$ for $\nu$, because the spectral gap $\lambda^{(2)}_{19} - \lambda^{(2)}_{18}$ is too small, see Table~\ref{ta:res}. Finally, we take $\nu = L^{(3)}_{17}$ and compute lower bounds on the first sixteen eigenvalues for the targeted domain $\Omega^{(4)} = \Omega$.

Table~\ref{ta:res} presents the analytically computed eigenvalues for $\Omega^{(0)}$ and two-sided bounds on eigenvalues obtained in the last adaptive step for domains $\Omega^{(1)}$, \dots, $\Omega^{(4)}$. We observe that the smaller eigenvalues are approximated with higher accuracy than the larger eigenvalues. This is caused by the recursive evaluation of the lower bound \eqref{eq:Kato}, see the description at the end of Section~\ref{se:lbounds}, %above Corollary~\ref{co:Katoboundalg},
and also by the fact that the higher eigenfunctions are more oscillatory and hence more difficult to approximate.

\begin{table}
\begin{tabular}{c|ccccc|c}
 & \multicolumn{5}{c|}{Bound $L_n$ with the homotopy method} & Combination \\
 & $\Omega^{(0)}$  & $\Omega^{(1)}$  & $\Omega^{(2)}$  & $\Omega^{(3)}$ & $\Omega^{(4)}$ & of $\ell_n$ and $L_n$ \\
\hline
$\lambda_{1}$ &   1.19753 & $1.351^{35}_{48}$    & $1.5974^{0}_{9}$    & $1.8301^{3}_{9}$    & $1.955^{76}_{81}$    & $1.955^{76}_{81}$ \\[1pt]
$\lambda_{2}$ &   1.79012 & $1.804^{18}_{27}$    & $1.8414^{3}_{8}$    & $1.8988^{0}_{3}$    & $1.960^{67}_{70}$    & $1.960^{67}_{70}$ \\[1pt]
$\lambda_{3}$ &   2.77778 & $2.926^{02}_{28}$    & $3.30^{181}_{211}$  & $4.11^{058}_{130}$  & $4.^{79998}_{80085}$ & $4.^{79998}_{80085}$ \\[1pt]
$\lambda_{4}$ &   4.16049 & $4.20^{787}_{916}$   & $4.33^{671}_{743}$  & $4.558^{22}_{86}$   & $4.829^{36}_{99}$    & $4.829^{36}_{99}$ \\[1pt]
$\lambda_{5}$ &   4.19753 & $4.62^{030}_{323}$   & $4.8^{7882}_{8018}$ & $4.971^{27}_{77}$   & $4.996^{50}_{91}$    & $4.996^{50}_{91}$ \\[1pt]
$\lambda_{6}$ &   4.79012 & $4.83^{558}_{720}$   & $4.91^{473}_{548}$  & $4.97^{384}_{409}$  & $4.996^{72}_{92}$    & $4.996^{72}_{92}$ \\[1pt]
$\lambda_{7}$ &   5.77778 & $6.02^{490}_{814}$   & $6.234^{14}_{84}$   & $7.08^{243}_{406}$  & $7.98^{460}_{714}$   & $7.98^{460}_{714}$ \\[1pt]
$\lambda_{8}$ &   5.93827 & $6.4^{0875}_{1330}$  & $7.40^{406}_{784}$  & $7.87^{152}_{502}$  & $7.98^{594}_{721}$   & $7.98^{594}_{721}$ \\[1pt]
$\lambda_{9}$ &   7.16049 & $7.32^{420}_{919}$   & $7.63^{297}_{553}$  & $7.88^{669}_{865}$  & $9.35^{272}_{752}$   & $9.35^{275}_{752}$ \\[1pt]
$\lambda_{10}$ &  8.11111 & $8.1^{8723}_{9010}$  & $8.37^{605}_{743}$  & $8.78^{241}_{440}$  & $9.5^{0753}_{1108}$  & $9.5^{0756}_{1108}$ \\[1pt]
$\lambda_{11}$ &  8.93827 & $9.^{38939}_{41074}$ & $9.9^{3808}_{5045}$ & $9.9^{8859}_{9296}$ & $9.99^{755}_{993}$   & $9.99^{757}_{993}$ \\[1pt]
$\lambda_{12}$ &  9.19753 & $9.7^{6375}_{8678}$  & $9.95^{150}_{772}$  & $9.99^{083}_{329}$  & $9.99^{874}_{993}$   & $9.99^{875}_{993}$ \\[1pt]
$\lambda_{13}$ &  9.79012 & $9.8^{5854}_{7295}$  & $10.7^{717}_{814}$  & $11.12^{41}_{65}$   & $12.^{6271}_{7319}$  & $12.^{6540}_{7319}$ \\[1pt]
$\lambda_{14}$ &  10.6790 & $10.^{6940}_{7023}$  & $10.9^{252}_{322}$  & $12.6^{194}_{408}$  & $12.^{7402}_{8306}$  & $12.^{7531}_{8306}$ \\[1pt]
$\lambda_{15}$ &  10.7778 & $11.^{3961}_{4158}$  & $12.0^{308}_{459}$  & $12.6^{669}_{832}$  & $12.^{8974}_{9693}$  & $12.^{8934}_{9693}$ \\[1pt]
$\lambda_{16}$ &  11.1111 & $12.0^{538}_{855}$   & $12.6^{336}_{505}$  & $12.^{6953}_{7073}$ & $12.9^{335}_{695}$   & $12.^{8937}_{9695}$ \\[1pt]
$\lambda_{17}$ &  12.1605 & $12.4^{238}_{417}$   & $12.8^{735}_{834}$  & $13.13^{16}_{86}$   &  \multicolumn{2}{c}{}  \\[1pt]
$\lambda_{18}$ &  13.6420 & $13.7^{351}_{576}$   & $14.0^{437}_{570}$  &  &  \multicolumn{2}{c}{}  \\[1pt]
$\lambda_{19}$ &  13.6790 & $13.^{8915}_{9068}$  & $14.6^{443}_{675}$  &  &  \multicolumn{2}{c}{}  \\[1pt]
$\lambda_{20}$ &  13.9383 & $15.^{1989}_{2422}$  &  &  &  \multicolumn{2}{c}{}  \\[1pt]
$\lambda_{21}$ &  16.1111 &  &  &  &  \multicolumn{2}{c}{}  \\
\end{tabular}
\caption{\label{ta:res}
Two-sided bounds computed in the last adaptive step. Digits in the upper/lower index indicate the lower/upper bound, respectively.
Eigenvalues on the rectangle $\Omega^{(0)}$ are known analytically.
Columns $\Omega^{(1)}$, \dots, $\Omega^{(4)}$ correspond to homotopy steps for lower bound $L_n$.
The last column presents the combination of lower bounds $\ell_n$ and $L_n$
%given by \eqref{eq:Wein} and \eqref{eq:Kato}
for the dumbbell shaped domain $\Omega^{(4)} = \Omega$.
% Eigenvalues during the homotopy method. Eigenvalues on the rectangle $\Omega^{(0)}$ are known analytically. Lower bounds on eigenvalues for domains $\Omega^{(1)}$, \dots, $\Omega^{(4)}$ are computed by \eqref{eq:Katoboundalg}, while the upper bounds by the standard finite element approximations \eqref{eq:EPFEM}.
% Column $\Omega^{(4)}$ presents the final results on the targeted dumbbell shaped domain.
% The last column presents the same eigenvalues computed by a combination of lower bounds \eqref{eq:Weinbound} and \eqref{eq:Katoboundalg}.
}
\end{table}
%
%     case 1
%       eiginfo.eigto = 20;
%       eiginfo.lownext = 16.111111; % lambda_21:  h0= 16.1111111
%     case 2
%       eiginfo.eigto = 19;
%       eiginfo.lownext = 15.103602; % lambda_20:  h1= 15.1036029322769
%     case 3
%       eiginfo.eigto = 17;
%       eiginfo.lownext = 13.974707; % lambda_18: h2 low=    13.9747070202251
%     case 4
%       eiginfo.eigto = 16;
%       eiginfo.lownext = 13.100068; % lambda_17: h3 low=    13.1000685569223

We note that the original goal was to compute two-sided bounds of the first ten eigenvalues, but it is advantageous to compute slightly more of them, because the performance of the lower bound $L_n$ depends on the size of the corresponding spectral gaps.
For example in the case of $\Omega^{(4)}$, we obtain $\nu$ in the reasonable large spectral gap between $\lambda_{16}$ and $\lambda_{17}$, which enables to find accurate lower bounds for the first ten eigenvalues.
% In the particular case of $\Omega^{(4)}$, there is a reasonable spectral gap between $\lambda_{16}$ and $\lambda_{17}$, but the chosen value of $\nu$ does not utilize its potential fully, because of the relatively large homotopy step. \cred{V predchozi vete nerozumim, co se chce rict.} However, this is sufficient to resolve well the large spectral gap between $\lambda_{12}$ and $\lambda_{13}$ and consequently to obtain accurate two-sided bounds on the first ten eigenvalues.

\begin{figure}
\includegraphics[width=0.485\textwidth]{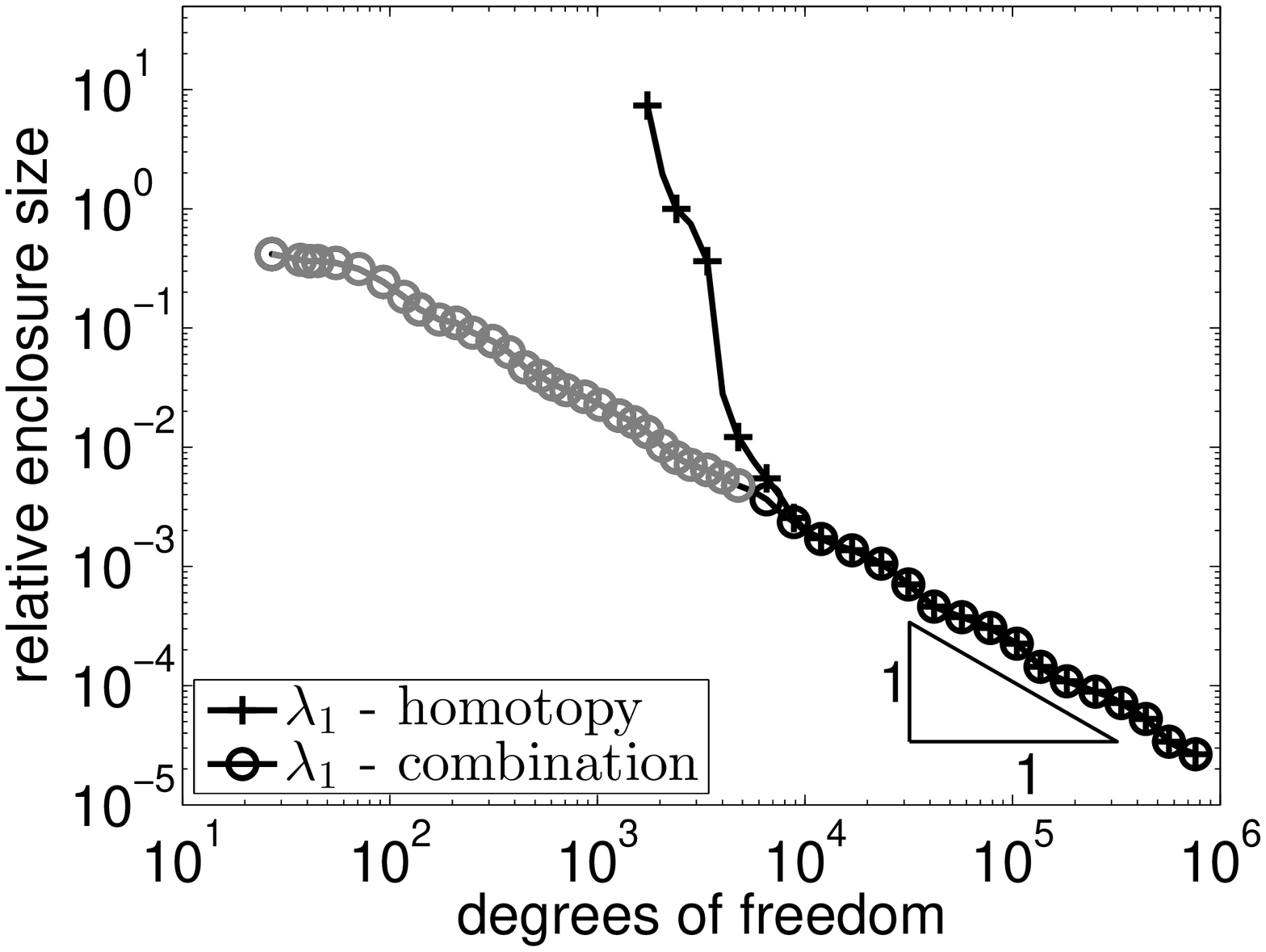}\quad
\includegraphics[width=0.485\textwidth]{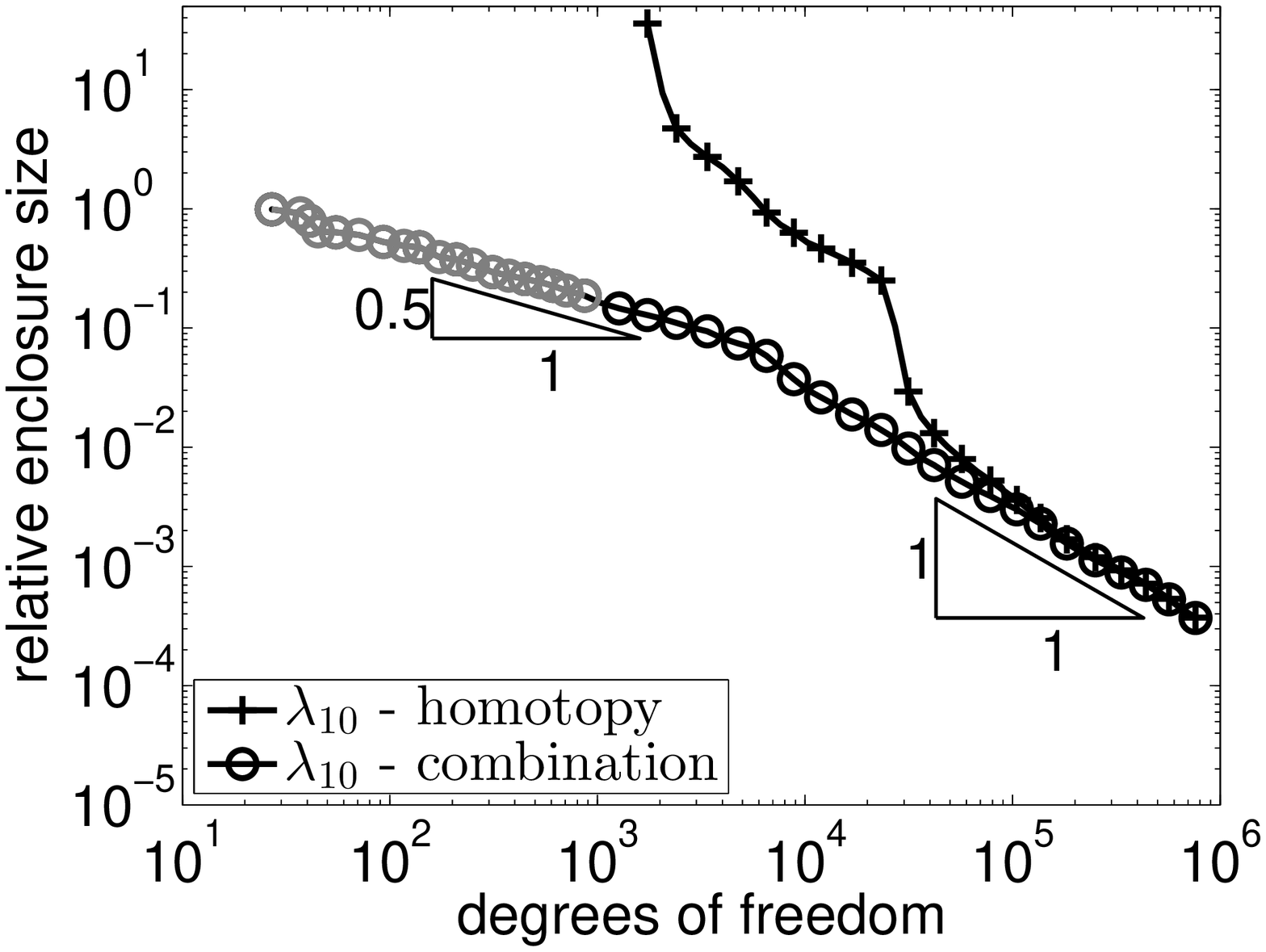}
\caption{\label{fi:conv}
Convergence curves for relative enclosure sizes $(\lambda_{h,n} - L_n)/L_n$ (homotopy) and $(\lambda_{h,n} - \underline{\lambda}_{h,n})/\underline{\lambda}_{h,n}$ (combination) for $n=1$ (left) and $n=10$ (right) with respect to the number of degrees of freedom.
Crosses correspond to the last homotopy step, circles to the combination of $\ell_n$ and $L_n$, see \eqref{eq:Wein} and \eqref{eq:Kato}.
Data points, where the best bounds indicate invalidity of closeness condition \eqref{eq:closest} are coloured in grey.
}
\end{figure}

\subsection*{Combination of $\ell_n$ and $L_n$.}
The rough idea is to compute the lower bound on $\lambda_n$ as $\max\{\ell_n, L_n \}$ and to use $\nu = \ell_{n+1}$ to evaluate $L_n$. Quantities $\ell_n$ and $L_n$ are given by \eqref{eq:Wein} and \eqref{eq:Kato}, respectively. This approach is computationally less demanding, because the expensive homotopy is not performed. On the other hand, this approach does not provide guaranteed lower bounds on eigenvalues due to the closeness condition \eqref{eq:closest}, which cannot be verified in practice. However, numerical experiments we performed indicate that the computed bounds are very reliable. In our numerical tests, we never found the computed lower bound to be above the best available upper bound. This applies not only to the situation when the problem is well resolved and the closeness condition \eqref{eq:closest} and the lower bound condition \eqref{eq:nucond} are likely to hold, but even to cases when the accuracy is low and these conditions are probably not satisfied. The reason is that these conditions are sufficient and if they are not satisfied then the computed lower bounds can still be below the true eigenvalues. Numerical results show that this is a very common situation.

Let us describe the algorithm in more details. It is based on the above specified adaptive algorithm with two changes. First, step \texttt{SOLVE} computes eigenpairs for $n=r,\dots,s+1$. Second,
step \texttt{ESTIMATE} is changed as follows:
\begin{itemize}
\item Solve local problems \eqref{min_prob_qa_1}--\eqref{min_prob_qa_2}, reconstruct the flux \eqref{eq:defq}, and compute error estimator $\eta_n$ by \eqref{eq:eta} for all eigenpairs $\lambda_{h,n}$, $u_{h,n}$, $n=r,\dots,s+1$.
\item Compute lower bounds $\ell_r,\dots,\ell_{s+1}$ by \eqref{eq:Wein}.
\item Set $\nu = \ell_{s+1}$, check the left-hand side inequality of condition \eqref{eq:nucond}, and compute lower bounds $L_r,\dots, L_s$ by the recursive application of \eqref{eq:Kato}.
\item The final lower bound is given by $\underline{\lambda}_{h,n} = \max\{\ell_n, L_n\}$ for all $n=r,\dots,s$.
\end{itemize}

%Then we refine the mesh adaptively as described above, construct the new mesh and repeat the process taking into account the best lower bounds $\underline{\lambda}_{h,n}$, $n=r,\dots,s$, we computed in previous adaptive steps.

The results of this procedure for the Laplace eigenvalue problem in the dumbbell shaped domain with $r=1$ and $s=15$ are provided in the last column of Table~\ref{ta:res}.
For illustration, Figure~\ref{fi:dumbbell} (right) shows the adaptively refined mesh after 20 adaptive steps.
In Figure~\ref{fi:conv}, we also present convergence curves for relative enclosure sizes $\Erelestn = (\lambda_{h,n} - \underline{\lambda}_{h,n})/\underline{\lambda}_{h,n}$.
Note that $\Erelestn$ bounds the true relative error: $(\lambda_{h,n} - \lambda_n)/\lambda_n \leq \Erelestn$ and, thus, provides a reliable information about the true error.
Figure~\ref{fi:conv} shows $\Erelestn$ for the first (left) and the tenth (right) eigenvalue. We observe that the combination of the adaptive algorithm and the quadratically convergent bound $L_n$ yields the expected order of convergence on sufficiently fine meshes.
We point out that suboptimal lower bound $\ell_n$ is useful on relatively coarse meshes, where it provides more accurate results than $L_n$. Indeed, convergence curves provided in Figure~\ref{fi:conv} typically start with a suboptimal slope corresponding to $\ell_n$ and as soon as the mesh is sufficiently fine, the optimal bound $L_n$ overcomes.
We also note that $L_n$ is not evaluated on those rough meshes where the left-hand side inequality in \eqref{eq:nucond} is not satisfied, see for example the curves for the homotopy method, where data points for small numbers of degrees of freedom are missing.
%the missing data points in Figure~\ref{fi:conv} for the homotopy method for small numbers of degrees of freedom
%the curves corresponding to the homotopy method. \cred{Jak je to z grafu videt?}
Finally, we used the best available lower bounds $\underline{\lambda}_{h,n}^\mathrm{best}$ in the last adaptive step and tested if $\lambda_{h,n} \leq \sqrt{ \underline{\lambda}_{h,n}^\mathrm{best} \underline{\lambda}_{h,n+1}^\mathrm{best} }$. Those points, where this test passed and closeness condition \eqref{eq:closest} is satisfied, are indicated as black circles in Figure~\ref{fi:conv}. Notice that even in grey data points, where closeness condition \eqref{eq:closest} is probably not valid, lower bound $\ell_n$ always produced a value below the best available lower bound and hence below the true eigenvalue.

% MOVE THIS COMMENT BELOW ???:
% We stop the adaptive algorithm as soon as the computed two-sided bounds on eigenvalues show that the relative error decreased below $\Ereltol = 10^{-2}$ for all of the first $s$ eigenvalues. More precisely, we define the relative error $\Ereln = |\lambda^{(1)}_n - \lambda^{(1)}_{h,n}|/\lambda^{(1)}_n$, its upper estimate
% $\Erelestn = (\lambda^{(1)}_{h,n} - L^{(1)}_n)/L^{(1)}_n$, where $L^{(1)}_n$ is the lower bound on $\lambda^{(1)}_n$ computed as described above, and stop the adaptive algorithm as soon as
% $$
%   \Erelestn \leq \Ereltol
%   \quad\text{for all } n=1,2,\dots,s.
% $$
% This stopping criterion guarantees that the true relative error $\Ereln$ is below the prescribed tolerance $\Ereltol$ for all $n=1,2,\dots,s$.

% As we already mentioned, lower bound \eqref{eq:Wein} is not guaranteed in practical computations, because the closeness condition \eqref{eq:closest} cannot be verified. However, we can increase our confidence by testing this condition a posteriori. We can use the most accurate lower bounds $\underline{\lambda}_{h,n}$ on eigenvalues computed so far and test if
% $$
%    \lambda_{h,n} \leq \sqrt{\underline{\lambda}_{h,n} \underline{\lambda}_{h,n+1}}.
% $$
% Figure~\ref{fi:conv} shows in black those data points, where this test passed, and where we have high confidence that the closeness condition \eqref{eq:closest} holds true and the computed bounds are really below the true eigenvalues.

Comparing results in Table~\ref{ta:res}, we observe that values in the last column are almost the same as values in the previous column. Thus, the combination of bounds $\ell_n$ and $L_n$ yielded almost the same lower bounds on eigenvalues although they are not guaranteed by the theory. Thus, if the bounds are not required to be guaranteed, we recommend to use the simple combination of bounds $\ell_n$ and $L_n$ rather than the homotopy method.

% TO DO:
%
% - result: the spectral gap between $\lambda_5$ and $\lambda_6$ is too small to be resolved, but we have guaranteed lower and upper bounds for the cluster of these two eigenvalues

\section{Conclusions}
\label{se:concl}

In this paper we generalized classical Weinstein's and Kato's bounds to the weak setting suitable for direct application of the finite element method. Needed guaranteed bounds on the representative of the residual are computed by the complementarity technique using the local flux reconstruction. Formulas \eqref{eq:Wein} and \eqref{eq:Kato} for lower bounds $\ell_n$ and $L_n$ are simple and explicit.
The bound $L_n$ is quadratically convergent and provides guaranteed lower bounds on all eigenvalues $\lambda_1, \lambda_2, \dots, \lambda_s$, provided a guaranteed lower bound $\nu$ on $\lambda_{s+1}$ is available. The computational efficiency of the bound $L_n$ depends on the size of the spectral gap $\lambda_{s+1}-\lambda_s$. Therefore, we recommend to first identify the index $s$ such that the spectral gap $\lambda_{s+1}-\lambda_s$ is relatively large and then use a guaranteed lower bound $\nu$ on $\lambda_{s+1}$ to estimate the smaller eigenvalues from below. For the optimal usage of this bound, we recommend the recursive algorithm described at the end of Section~\ref{se:lbounds}.%above Corollary~\ref{co:Katoboundalg}.

The guaranteed lower bound on $\lambda_{s+1}$ can be computed, for example, by the homotopy method as we illustrated in Section~\ref{se:numex}. Numerical experiments, however, show that $L_n$ need not be sufficiently accurate on rough meshes. Therefore, we recommend to compute both $\ell_n$ and $L_n$ and use the larger of them. We note that $\ell_n$ can be evaluated very cheaply as soon as $\eta_n$ is available.
The bound $\ell_n$ is guaranteed to be below the corresponding exact eigenvalue $\lambda_n$ if the (upper) finite element approximation $\lambda_{h,n}$ is in a sense closer to $\lambda_n$ then to $\lambda_{n+1}$, see \eqref{eq:closest}. This condition can be verified if guaranteed lower bounds on $\lambda_n$ and $\lambda_{n+1}$ are available, for example using $L_n$ and $L_{n+1}$.

If the particular application does not require the computed lower bounds to be guaranteed then the homotopy method is not necessary. We can use $\ell_n$ as the first lower bounds and then employ them to compute more accurate bounds $L_n$. We would like to emphasize that the conditions guaranteeing that $\ell_n$ and $L_n$ are below the corresponding exact eigenvalues are sufficient conditions only. Thus, if they are not satisfied then the computed bounds can still be below the exact values. Performed numerical experiments confirm that this is actually a very common case in practice.

In our future research we will focus on the proof of the local efficiency of the estimator given by Corollary~\ref{co:normw} with the local flux reconstruction satisfying \eqref{eq:defq} and \eqref{min_prob_qa_1}--\eqref{min_prob_qa_2}. We then plan to use this local efficiency result to prove the convergence of the corresponding adaptive algorithm.

% %    Bibliographies can be prepared with BibTeX using amsplain,
% %    amsalpha, or (for "historical" overviews) natbib style.
% \bibliographystyle{amsplain}
% %    Insert the bibliography data here.
% \bibliography{bibl}

\begin{thebibliography}{10}

\bibitem{Adams_Sob_spaces_03}
Robert~A. Adams and John J.~F. Fournier, \emph{Sobolev spaces},
  Elsevier/Academic Press, Amsterdam, 2003. \MR{2424078 (2009e:46025)}

\bibitem{AinOde:2000}
Mark Ainsworth and J.~Tinsley Oden, \emph{A posteriori error estimation in
  finite element analysis}, Wiley-Interscience [John Wiley \& Sons], New York,
  2000. \MR{MR1885308 (2003b:65001)}

\bibitem{AinVej2014}
Mark Ainsworth and Tom\'a\v{s} Vejchodsk\'y, \emph{Robust error bounds for
  finite element approximation of reaction-diffusion problems with non-constant
  reaction coefficient in arbitrary space dimension}, Comput. Methods Appl.
  Mech. Engrg. \textbf{281} (2014), 184--199.

\bibitem{AinVej2014corr}
\bysame, \emph{Corrigendum to ``{R}obust error bounds for finite element
  approximation of reaction--diffusion problems with non-constant reaction
  coefficient in arbitrary space dimension''}, Comput. Methods Appl. Mech.
  Engrg. \textbf{299} (2016), 143--143.

\bibitem{AndRac:2012}
Andrey Andreev and Milena Racheva, \emph{Two-sided bounds of eigenvalues of
  second-and fourth-order elliptic operators}, Appl. Math. \textbf{59} (2014),
  no.~4, 371--390.

\bibitem{BabOsb:1991}
Ivo Babu{\v{s}}ka and John~E. Osborn, \emph{Eigenvalue problems}, Handbook of
  numerical analysis, {V}ol.~{II}, North-Holland, Amsterdam, 1991,
  pp.~641--787. \MR{1115240}

\bibitem{Barrenechea2014}
G.~R. Barrenechea, L.~Boulton, and N.~Boussa{\"{\i}}d, \emph{Finite element
  eigenvalue enclosures for the {M}axwell operator}, SIAM J. Sci. Comput.
  \textbf{36} (2014), no.~6, A2887--A2906. \MR{3285902}

\bibitem{BehnkeGoerish1994}
H.~Behnke and F.~Goerisch, \emph{Inclusions for eigenvalues of selfadjoint
  problems}, Topics in validated computations ({O}ldenburg, 1993), Stud.
  Comput. Math., vol.~5, North-Holland, Amsterdam, 1994, pp.~277--322.
  \MR{1318957}

\bibitem{Biegert:2009}
Markus Biegert, \emph{On traces of {S}obolev functions on the boundary of
  extension domains}, Proc. Amer. Math. Soc. \textbf{137} (2009), no.~12,
  4169--4176. \MR{2538577 (2010m:46045)}

\bibitem{Braess2013}
Dietrich {Braess}, \emph{{Finite Elemente. Theorie, schnelle L\"oser und
  Anwendungen in der Elastizit\"atstheorie.}}, 5th revised ed. ed., Berlin:
  Springer Spektrum, 2013 (German).

\bibitem{BraSch:2008}
Dietrich Braess and Joachim Sch{\"o}berl, \emph{Equilibrated residual error
  estimator for edge elements}, Math. Comp. \textbf{77} (2008), no.~262,
  651--672. \MR{2373174 (2008m:65313)}

\bibitem{BreFor:1991}
Franco Brezzi and Michel Fortin, \emph{Mixed and hybrid finite element
  methods}, Springer Series in Computational Mathematics, vol.~15,
  Springer-Verlag, New York, 1991. \MR{1115205 (92d:65187)}

\bibitem{CarGal2014}
Carsten Carstensen and Dietmar Gallistl, \emph{Guaranteed lower eigenvalue
  bounds for the biharmonic equation}, Numer. Math. \textbf{126} (2014), no.~1,
  33--51. \MR{3149071}

\bibitem{CarGed2014}
Carsten Carstensen and Joscha Gedicke, \emph{Guaranteed lower bounds for
  eigenvalues}, Math. Comp. \textbf{83} (2014), no.~290, 2605--2629.
  \MR{3246802}

\bibitem{Chatelin1983}
Fran{\c{c}}oise Chatelin, \emph{Spectral approximation of linear operators},
  Academic Press, Inc. [Harcourt Brace Jovanovich, Publishers], New York, 1983.
  \MR{716134}

\bibitem{Dol_Ern_Vohr_hp_refinement_strategies_polyn_rob_AEE}
V{\'{\i}}t Dolej{\v{s}}{\'{\i}}, Alexandre Ern, and Martin Vohral{\'{\i}}k,
  \emph{{$hp$}-adaptation driven by polynomial-degree-robust a posteriori error
  estimates for elliptic problems}, SIAM J. Sci. Comput. \textbf{38} (2016),
  no.~5, A3220--A3246. \MR{3556071}

\bibitem{Dorfler:1996}
Willy D{\"o}rfler, \emph{A convergent adaptive algorithm for {P}oisson's
  equation}, SIAM J. Numer. Anal. \textbf{33} (1996), no.~3, 1106--1124.
  \MR{MR1393904 (97e:65139)}

\bibitem{Ern_Voh_adpt_IN_13}
Alexandre Ern and Martin Vohral{\'{\i}}k, \emph{Adaptive inexact {N}ewton
  methods with a posteriori stopping criteria for nonlinear diffusion {PDE}s},
  SIAM J. Sci. Comput. \textbf{35} (2013), no.~4, A1761--A1791. \MR{3072765}

\bibitem{Gaal_Lin_anal_repres_theo_73}
Steven~A. Gaal, \emph{{Linear analysis and representation theory}},
  {Berlin-Heidelberg-New York, Springer-Verlag}, 1973 (English).

\bibitem{GruOva2009}
Luka Grubi{\v{s}}i{\'c} and Jeffrey~S. Ovall, \emph{On estimators for
  eigenvalue/eigenvector approximations}, Math. Comp. \textbf{78} (2009),
  no.~266, 739--770. \MR{2476558}

\bibitem{HuHuaLin2014}
Jun Hu, Yunqing Huang, and Qun Lin, \emph{Lower bounds for eigenvalues of
  elliptic operators: by nonconforming finite element methods}, J. Sci. Comput.
  \textbf{61} (2014), no.~1, 196--221. \MR{3254372}

\bibitem{HuHuaShe2015}
Jun {Hu}, Yunqing {Huang}, and Quan {Shen}, \emph{{Constructing both lower and
  upper bounds for the eigenvalues of elliptic operators by nonconforming
  finite element methods.}}, {Numer. Math.} \textbf{131} (2015), no.~2,
  273--302 (English).

\bibitem{Kato1949}
Tosio Kato, \emph{On the upper and lower bounds of eigenvalues}, J. Phys. Soc.
  Japan \textbf{4} (1949), 334--339. \MR{0038738}

\bibitem{Kobayashi2015}
K.~Kobayashi, \emph{On the interpolation constants over triangular elements},
  Applications of Mathematics 2015 (J.~Brandts, S.~Korotov,
  M.~K{\v{r}}{\'\i}{\v{z}}ek, K.~Segeth, J.~{\v{S}}{\'\i}stek, and
  T.~Vejchodsk{\'y}, eds.), Institute of Mathematics CAS, Prague, 2015,
  pp.~110--124.

\bibitem{Kuf_John_Fucik_Function_spaces}
Alois Kufner, Old{\v{r}}ich John, and Svatopluk Fu{\v{c}}{\'{\i}}k,
  \emph{Function spaces}, Noordhoff International Publishing, Leyden, 1977.
  \MR{0482102 (58 \#2189)}

\bibitem{KutSig:1978}
James~R. Kuttler and Vincent~G. Sigillito, \emph{Bounding eigenvalues of
  elliptic operators}, SIAM J. Math. Anal. \textbf{9} (1978), no.~4, 768--778.
  \MR{0492948 (58 \#11996)}

\bibitem{Kuz_Rep_Guar_low_bound_smal_eig_elip_13}
Yu.~A. Kuznetsov and S.~I. Repin, \emph{Guaranteed lower bounds of the smallest
  eigenvalues of elliptic differential operators}, J. Numer. Math. \textbf{21}
  (2013), no.~2, 135--156. \MR{3101022}

\bibitem{LinLuoXie2014}
Qun Lin, Fusheng Luo, and Hehu Xie, \emph{A posterior error estimator and lower
  bound of a nonconforming finite element method}, J. Comput. Appl. Math.
  \textbf{265} (2014), 243--254. \MR{3176271}

\bibitem{LinXieLuoLiYan:2010}
Qun Lin, Hehu Xie, Fusheng Luo, Yu~Li, and Yidu Yang, \emph{Stokes eigenvalue
  approximations from below with nonconforming mixed finite element methods},
  Math. Pract. Theory \textbf{40} (2010), no.~19, 157--168. \MR{2768711}

\bibitem{Liu2015}
Xuefeng Liu, \emph{A framework of verified eigenvalue bounds for self-adjoint
  differential operators}, Appl. Math. Comput. \textbf{267} (2015), 341--355.
  \MR{3399052}

\bibitem{LiuOis2013}
Xuefeng Liu and Shin'ichi Oishi, \emph{Verified eigenvalue evaluation for the
  {L}aplacian over polygonal domains of arbitrary shape}, SIAM J. Numer. Anal.
  \textbf{51} (2013), no.~3, 1634--1654. \MR{3061473}

\bibitem{LuoLinXie:2012}
Fusheng Luo, Qun Lin, and Hehu Xie, \emph{Computing the lower and upper bounds
  of {L}aplace eigenvalue problem: by combining conforming and nonconforming
  finite element methods}, Sci. China Math. \textbf{55} (2012), no.~5,
  1069--1082. \MR{2912496}

\bibitem{Payne_Weinberger_opti_Poinc_ineq_conv_domains_1960}
L.E. Payne and H.F. Weinberger, \emph{{An optimal Poincar\'e inequality for
  convex domains}}, Arch. Ration. Mech. Anal. \textbf{5} (1960), 286--292
  (English).

\bibitem{Plum1990}
Michael {Plum}, \emph{{Eigenvalue inclusions for second-order ordinary
  differential operators by a numerical homotopy method.}}, {Z. Angew. Math.
  Phys.} \textbf{41} (1990), no.~2, 205--226 (English).

\bibitem{Plum1991}
Michael Plum, \emph{Bounds for eigenvalues of second-order elliptic
  differential operators}, Z. Angew. Math. Phys. \textbf{42} (1991), no.~6,
  848--863. \MR{1140697}

\bibitem{Plum1997}
\bysame, \emph{Guaranteed numerical bounds for eigenvalues}, Spectral theory
  and computational methods of {S}turm-{L}iouville problems ({K}noxville, {TN},
  1996), Lecture Notes in Pure and Appl. Math., vol. 191, Dekker, New York,
  1997, pp.~313--332. \MR{1460558}

\bibitem{Quar_Val_Num_appr_PDE_94}
Alfio Quarteroni and Alberto Valli, \emph{Numerical approximation of partial
  differential equations}, Springer Series in Computational Mathematics,
  vol.~23, Springer-Verlag, Berlin, 1994. \MR{MR1299729 (95i:65005)}

\bibitem{Rannacher:1979}
Rolf Rannacher, \emph{Nonconforming finite element methods for eigenvalue
  problems in linear plate theory}, Numer. Math. \textbf{33} (1979), no.~1,
  23--42. \MR{545740 (80i:65124)}

\bibitem{Repin2012}
S.~I. Repin, \emph{Computable majorants of constants in the {P}oincar\'e and
  {F}riedrichs inequalities}, J. Math. Sci. (N. Y.) \textbf{186} (2012), no.~2,
  307--321, Problems in mathematical analysis. No. 66. \MR{3098308}

\bibitem{Rep:2008}
Sergey Repin, \emph{A posteriori estimates for partial differential equations},
  Walter de Gruyter GmbH \& Co. KG, Berlin, 2008. \MR{MR2458008}

\bibitem{SchSie2005}
Alfred Schmidt and Kunibert~G. Siebert, \emph{Design of adaptive finite element
  software}, Lecture Notes in Computational Science and Engineering, vol.~42,
  Springer-Verlag, Berlin, 2005, The finite element toolbox ALBERTA, With 1
  CD-ROM (Unix/Linux). \MR{2127659 (2005i:65003)}

\bibitem{Seb_Vejch_2sidedb_eigen_Fr_Poin_trace_14}
Ivana {\v{S}}ebestov{\'a} and Tom{\'a}{\v{s}} Vejchodsk{\'y}, \emph{Two-sided
  bounds for eigenvalues of differential operators with applications to
  {F}riedrichs, {P}oincar\'e, trace, and similar constants}, SIAM J. Numer.
  Anal. \textbf{52} (2014), no.~1, 308--329. \MR{3163245}

\bibitem{Sigillito:1977}
Vincent~G. Sigillito, \emph{Explicit a priori inequalities with applications to
  boundary value problems}, Pitman Publishing, London-San Francisco,
  Calif.-Melbourne, 1977. \MR{0499654 (58 \#17459)}

\bibitem{Synge:1957}
J.~L. Synge, \emph{The hypercircle in mathematical physics: a method for the
  approximate solution of boundary value problems}, Cambridge University Press,
  New York, 1957. \MR{MR0097605 (20 \#4073)}

\bibitem{TreBet2006}
Lloyd~N. Trefethen and Timo Betcke, \emph{Computed eigenmodes of planar
  regions}, Recent advances in differential equations and mathematical physics,
  Contemp. Math., vol. 412, Amer. Math. Soc., Providence, RI, 2006,
  pp.~297--314. \MR{2259116}

\bibitem{Complement:2010}
Tom{\'a}{\v s} Vejchodsk{\'y}, \emph{Complementarity based a~posteriori error
  estimates and their properties}, Math. Comput. Simulation \textbf{82} (2012),
  no.~10, 2033--2046 (English).

\bibitem{systemaee:2010}
\bysame, \emph{Complementary error bounds for elliptic systems and
  applications}, Appl. Math. Comput. \textbf{219} (2013), no.~13, 7194--7205
  (English).

\bibitem{Voh_un_apr_apost_MFE_10}
Martin Vohral{\'{\i}}k, \emph{Unified primal formulation-based a priori and a
  posteriori error analysis of mixed finite element methods}, Math. Comp.
  \textbf{79} (2010), no.~272, 2001--2032. \MR{2684353 (2012c:65209)}

\bibitem{YanZhaLin:2010}
Yidu Yang, Zhimin Zhang, and Fubiao Lin, \emph{Eigenvalue approximation from
  below using non-conforming finite elements}, Sci. China Math. \textbf{53}
  (2010), no.~1, 137--150. \MR{2594754 (2011d:65341)}

\end{thebibliography}

\providecommand{\bysame}{\leavevmode\hbox to3em{\hrulefill}\thinspace}
\providecommand{\MR}{\relax\ifhmode\unskip\space\fi MR }
\providecommand{\MRhref}[2]{%
  \href{http://www.ams.org/mathscinet-getitem?mr=#1}{#2}
}
\providecommand{\href}[2]{#2}

\end{document}